\documentclass[a4paper,12pt,oneside,openbib]{article}

\usepackage{mathrsfs}
\usepackage{amsmath}
\usepackage{moreverb}
\usepackage{vmargin}
\usepackage{pdfpages}
\usepackage{bm}
\usepackage{bbm}
\usepackage{latexsym,amsfonts,amssymb,amscd}
\usepackage[T1]{fontenc}
\usepackage{amsthm}
\usepackage{tcolorbox}
\usepackage{dsfont} 
\usepackage{etoolbox}
\usepackage[square, numbers, comma, sort&compress]{natbib} 
\usepackage{verbatim}
\usepackage{bm}
\usepackage[T1]{fontenc}
\usepackage{tikz}
\usepackage[english]{babel}
\usepackage{pdfpages} 
\usepackage{moreverb}
\usepackage{vmargin}
\usepackage{tcolorbox}
\usepackage{abstract}
\numberwithin{equation}{section} 
\newcommand{\RomanNumeralCaps}[1]
 {\MakeUppercase{\romannumeral #1}}
 \newcommand{\R}{{ I\!\!R}}
 \newcommand{\N}{{I\!\!N}}

\newtheorem{theorem}{Theorem}[section]
\newtheorem{corollary}[theorem]{Corollary}
\newtheorem{lemma}[theorem]{Lemma}
\newtheorem{proposition}[theorem]{Proposition}
\newtheorem{definition}[theorem]{Definition}
\newtheorem{remark}[theorem]{Remark}

\usepackage{authblk}

\title{\bf On the Uniqueness and Orbital Stability of Slow and Fast Solitary Wave Solutions of the Benjamin Equation}

\date{}

\author[1]{May ABDALLAH 
\thanks{May\textunderscore{}i\textunderscore{}abdallah@outlook.com}}

\author[2]{Mohamad DARWICH \thanks{Mohamad.Darwich@lmpt.univ-tours.fr}}

\author[3] {Luc MOLINET \thanks{Luc.Molinet@lmpt.univ-tours.fr}}

\affil[1,3] {Institut Denis Poisson, Universit\'e de Tours, Universit\'e d'Orl\'eans, CNRS, Parc Grandmont, 37200 Tours, France}
\affil[1,2] {Laboratory of Mathematics-EDST, Department of Mathematics, Faculty of Sciences \RomanNumeralCaps{1}, Lebanese University, Beirut, Lebanon}

\begin{document}

\maketitle

\begin{abstract}
This paper is devoted to the study of existence and properties of solitary waves of the Benjamin equation. The studied equation includes a parameter $\gamma$ in front of the Benjamin-Ono term. We show the existence, uniqueness, decay and orbital stability of solitary wave solutions obtained as a solution to a certain minimization problem, associated either with high speeds without a sign condition on the parameter $\gamma$ or with low speeds for the appropriate sign.
\end{abstract}
\vspace{0.3cm}
\hspace{0.8cm}
\hspace*{3cm} -Constraint Minimization Problem - Orbital Stability
\
\section{Introduction}\label{section 1}

We consider the Benjamin equation 
\begin{equation}\label{maineq}
\partial_t u +u_{xxx}  +  \gamma \, \mathcal{H}   u_{xx} + u \partial_x u = 0, \quad (t,x)\in \R^2 ,
\end{equation}
where $ u $ is a real-valued function, $\mathcal{H}$ denotes the Hilbert transform and  $ \gamma\in \R^*$.
Here, the Hilbert transform $ \mathcal{H} $ is defined by 
\begin{equation} \label{hilbert}
 \mathcal{H} v\,  (x) = p.v. \frac{1}{\pi} \int_{\mathbb{R}} \frac{v(y)}{y-x} \ dy = \frac{1}{\pi} \lim_{\varepsilon \rightarrow 0} \int_{|y-x|> \varepsilon} \frac{v(y)}{y-x} \ dy .
\end{equation}
With this notation, we note that $\displaystyle \int u_x \mathcal{H} u = \int |D_x^\frac{1}{2} u|^2 = \|u\|_{\dot{H}^\frac{1}{2}}^2$. 
Moreover, $\mathcal{H}$ can be written as  a Fourier multiplier operator given by $$\widehat{\mathcal{H}(u)}(\xi) = i \ \text{sgn} (\xi) \ \hat{u}(\xi) \; .
$$
 In the case $ \gamma<0 $, this equation has been derived in \cite{Benjamin 1992} (see also  \cite {Benjamin 1996} and \cite{ABR}) as an asymptotic model for  the propagation    of long interfacial waves when the surface tension cannot be safely ignored. In this context,  $ u $ represents the vertical displacement of the interface between the two fluids. Actually, physically, it could be reduced to the Korteweg-de Vries (KdV) equation
$$ u_t +  u_{xxx} +  u u_x =0$$
when the dispersive effects are dominated by the surface tension, and to the Benjamin-Ono (BO) equation 
$$ u_t - \mathcal{H} u_{xx} + u u_x =0$$
when the surface tension is negligible.

Equation \eqref{maineq} enjoys the following conservation laws
\begin{equation}\label{energy}
\text{Energy:} \quad E(u(t)) = \int_{\mathbb{R}} \bigg(  \frac{1}{2} (\partial_x u)^2+ \frac{\gamma}{2} (D_x^{\frac{1}{2}} u )^2 - \frac{1}{6} u^3 \bigg) dx \quad  = E(u(0)),
\end{equation}
and 
$$ \text{Mass:} \quad M(u(t)) = \frac{1}{2} \ \int_{\mathbb{R}} u^2 dx \quad = M(u(0)),$$  and reduces to KdV when $ \gamma=0 $. It is worth noticing that the physical case, associated with the condition $ \gamma<0 $, is  where the KdV operator $ \partial_x^3 $ and the Benjamin-Ono (BO) operator $\mathcal{H} \partial_x^2 $ compete together. In particular, the quadratic part of the energy is then not positively defined. Note also that, whereas  the KdV and BO equations are completely integrable systems,  the Benjamin equation is not known to be integrable.

 It is well known that the Cauchy problem associated with \eqref{maineq} is globally well-posed in the energy space $H^1(\mathbb{R})$ (see \cite{Kenig Ponce Vega 1991}) and even in $ H^s(\mathbb{R}) $ for $s\ge 0$ (see \cite{linares}). In this work, we are mainly interested in the existence and properties of traveling wave solutions of \eqref{maineq}, that is, solutions of the form 
\begin{equation}\label{twsol}
u(x,t) = \varphi(x-ct) 
\end{equation}
with $ c>0 $ and $ \varphi (x)\to 0 $ as $ |x|\to +\infty$. Substituting \eqref{twsol} into equation \eqref{maineq}, followed by integration on $\mathbb{R}$, we infer that  such profiles $ \varphi $ satisfy
\begin{equation}\label{Phimaineq}
c\ \varphi - \ \varphi'' + \gamma \mathcal{H} \varphi' = \frac{1}{2} \varphi ^2.
\end{equation}

The existence, spatial asymptotics and orbital stability of solitary waves to \eqref{maineq} have already been studied in several papers. In \cite{Benjamin 1992} and \cite{Benjamin 1996}, a Leray-Schauder degree theory approach is used to show the existence of solitary wave profiles to \eqref{maineq} with $\gamma<0 $ for a speed $ c >\gamma^2/4 $. In \cite{ABR}, the implicit function theorem together with the existence and stability result in the case $ \gamma=0 $ (KdV equation) is used to prove the existence and uniqueness of  even  solitary wave profiles  to \eqref{maineq} for  $ |\gamma| $ small enough. This argument together with classical perturbation theory and the criterion for orbital stability established in  the seminal paper \cite{Bona et al 1987} enables then to obtain the orbital stability of these solitary waves for $ |\gamma| $ small enough. Finally,  variational approaches are used in \cite{bonachen} and  \cite{Angulo} to prove the existence of solitary waves for 
 $ c>\gamma^2/4$ whenever $ \gamma<0 $. In \cite{bonachen}, an optimal  decay of the solitary waves is also provided whereas in \cite{Angulo} the stability of the set of ground states is established for the same range of speed. 
 
 Our goal in this work is to complete  these results in some directions and even give another proof for some of them. We first prove the existence of solitary wave profiles to \eqref{maineq} for $ c> \gamma^2/4 $ in the case $ \gamma<0 $ and for $c>0 $ in the case $ \gamma>0 $. To do this, we use a variational approach as in \cite{Angulo} but instead of minimizing the Energy with the Mass as a constraint, following \cite{bonachen} (see also \cite{Levandosky 1999}), we minimize the quadratic part of the action functional $ E+c M $ under the constraint of its cubic part. This has the advantage of being really simpler but has the disadvantage of not directly leading to the orbital stability of the set of ground states. Actually in this work, we consider only the uniqueness and orbital stability for fast or slow solitary waves. First, by a dilation argument, we notice that this is equivalent to looking at solutions of \eqref{Phimaineq2c} below with either $ \alpha=1 $ and small $ |\beta| $, or $\beta=1 $ and small $ \alpha>0 $. Then, we follow  the ideas of \cite{Kenig Martel Robbiano 2011}. We prove the convergence of our ground state solutions of  \eqref{Phimaineq2} to the KdV soliton profile of speed $1$ in  case $ \alpha=1 $ and $ \beta\to 0 $, and to the BO soliton profile of speed $ 1$ in case $\beta=1 $ and $ \alpha\searrow 0 $. Later, we establish the uniqueness and orbital stability in $H^1(\mathbb{R}) $ of our ground state solutions in both cases when either $ \beta  $ or $\alpha>0$ are close enough to $ 0 $ by a perturbation argument making use of known results on the coercivity of the second derivative of the action functional associated respectively with the KdV and the BO equation. Note that this type of approaches can also be found in \cite{Kabakouala Molinet 2018} and \cite{Darwich 2019}. Finally, we follow arguments developed in \cite{bonali} to prove decay estimates for even  $ H^1$-solutions to \eqref{Phimaineq} as well as their first derivatives. Briefly, compared to existing results, our main  contributions  are essentially  the uniqueness and orbital stability of the slow ground state solutions in the case $ \gamma>0 $. Note also that, in contrary to the arguments in \cite{ABR},  our arguments have the advantage of being applicable for both slow and fast solitary waves.
 
   To end this introduction, we would like to emphasize that, as  noticed in \cite{Klein et al 2023}, classical Pohozaev identities lead to the non existence of solitary wave with speed $ c\le -\gamma^2/5$ in the case $ \gamma<0 $ (and $c\le 0 $ in the case $ \gamma>0$). Therefore, the question of existence of solitary waves to \eqref{maineq} with $\gamma<0 $ and speed $c\in ] -\gamma^2/5,\gamma^2/4]$ is still open.

\subsection{Main Results}

By using dilation symmetries, we will connect the study of  \eqref{maineq} to the study of the following equation 
\begin{equation}\label{maineq2}
\partial_t u + \alpha \partial_x^3 u  +\beta  \mathcal{H} \partial_x^2 u + u \partial_x u = 0
\end{equation}
with either $ \alpha=1 $ and $ |\beta| \ll 1 $ or $ \beta=1 $ and $ 0<\alpha\ll 1 $. Considering traveling wave solutions with speed $ c$, $\varphi_c(x-c \, t )$ of \eqref{maineq2}, we observe that they solve
\begin{equation}\label{Phimaineq2c}
c \varphi - \alpha \varphi'' - \beta \mathcal{H} \varphi' = \frac{1}{2} \varphi ^2.
\end{equation} 
Note that any solution $\varphi_{c, \alpha, \beta} $ to \eqref{Phimaineq2c} will have the index $\{c, \alpha, \beta\}$ referring to different values of $c, \ \alpha,$ and $\beta$, with $\varphi_{\alpha, \beta}$ the particular solution of \eqref{Phimaineq2} corresponding to $c=1$. Following  \cite{Levandosky 1999}, we set 
\begin{equation}\label{I}
I(\varphi) = I_{c,\alpha,\beta}(\varphi)=\frac{1}{2} \int_{\mathbb{R}} \bigg( \alpha(\varphi_x)^2+\beta (D_x^{\frac{1}{2}} \varphi)^2+c \varphi^2  \bigg) dx\; ,
\end{equation}
\begin{equation}\label{K}
K(\varphi) =  \frac{1}{6}\int_{\mathbb{R}} \varphi^3 dx \ ,
\end{equation}
 and consider the constraint-minimization  problem
\begin{equation}\label{minpb}
M^\lambda=M^\lambda_{c,\alpha,\beta} = \inf \{ I_{c,\alpha,\beta}(\varphi); \ \varphi \in H^1_e \ \text{such \ that} \ K(\varphi)= \lambda \}\; , 
\end{equation}
where $ \lambda>0 $ and $H_e^s (\mathbb{R})$ is the closed subspace of $ H^s(\R) $ defined for $s \geq 0$ by 
$$ H_e^s (\mathbb{R})= \{ u \in H^s(\mathbb{R})\, / \ u(-y)= u(y) \ \forall \ y \in \mathbb{R} \}\; .
$$
For our purpose, we will mainly work with traveling wave solutions with speed $ c=1 $ that satisfy
\begin{equation}\label{Phimaineq2}
 \varphi - \alpha \varphi'' - \beta \mathcal{H} \varphi' = \frac{1}{2} \varphi ^2.
\end{equation} 
Note that upon integrating this equation against $ \varphi$, it can be checked that any $ H^1(\mathbb{R})$-solution to \eqref{Phimaineq2} satisfies 
\begin{equation} \label{gssol}
2 \ I_{1,\alpha, \beta} (\varphi) = 3 \  K(\varphi) .
\end{equation}

Before stating our main result, let's define some of the notions to be used throughout this work.
\begin{definition}{(Ground state solution)} \label{def gs}
An even ground state solution of \eqref{Phimaineq} is an $H^1(\R)$-solution to \eqref{Phimaineq} that solves the constraint minimization problem $M^{\lambda}_{c,1,\gamma} $ for some $\lambda >0$.
\end{definition}

\begin{definition}{(Orbital stability)}
Let $v \in H^1(\mathbb{R})$ be a solution to \eqref{Phimaineq}. We say that $v$ is orbitally stable in $H^1(\mathbb{R})$ if for any $\delta >0$, there exists $\eta >0$ such that for all initial data $u_0 \in H^1(\mathbb{R})$ satisfying
$$ \| u_0 - v\|_{H^1} \leq \eta ,$$
the solution $u \in C(\mathbb{R}_+, H^1)$ associated to the initial data $u_0$ in turn satisfies
$$ \sup_{t \in \mathbb{R}_+} \inf_{z \in \mathbb{R}} \|u(t, . +z) - v\|_{H^1} \leq \delta. $$
\end{definition}
We are now in position to state our main results. 
\begin{theorem}\label{Thm1}[Existence, uniqueness, asymptotic behavior and orbital stability of traveling waves]\\
 \eqref{maineq} possesses traveling waves solutions with any speed $ c>0 $ whenever $ \gamma>0 $ and with any speed $ c> \gamma^2/4 $ whenever $\gamma<0$. The profile $ \varphi $ of this traveling wave is an even function belonging to $ H^\infty(\R) $ and satisfying 
  $$
  x^2 |\varphi|+ |x^3| (|\varphi'|+|\varphi''| ) \in L^\infty(\mathbb{R}) \; .
  $$
 Moreover, there exists $ r_2>r_1>0 $ such that :
\begin{itemize}
\item If $ \gamma>0 $, then for $ c\in ]0,r_1  \gamma^2[ \cup ]r_2  \gamma^2,+\infty[ $ \ this profile is the unique even ground state solution to \eqref{maineq} and is orbitally stable. 
\item  If $ \gamma<0 $, then for $ c\in ]r_2  \gamma^2,+\infty [ $ \ this profile is the unique even ground state solution to  \eqref{maineq} and is orbitally stable. 
\end{itemize}
\end{theorem} 
\begin{remark}
In the case $ \gamma>0 $, it follows from \cite{lopezmaris} that any solution of the constraint minimization problem 
$$
\widetilde{M}^\lambda_{c,\alpha,\beta}= \inf \{ I_{c,\alpha,\beta}(\varphi); \ \varphi \in H^1(\mathbb{R})  \ \text{such that} \ K(\varphi)= \lambda \}\
$$
has to be an even function. Therefore, the  uniqueness result still holds when  replacing $M^\lambda_{c,1,\gamma} $ by $\widetilde{M}^\lambda_{c,1,\gamma}$. In other words, we do not need to minimize in the set of even functions to get the uniqueness.
\end{remark}

\subsection{Organization of the paper}

The organization of this paper will be as follows: in Section \ref{section 2}, we prove the existence of minimizers to \eqref{minpb} by using the concentration-compactness lemma. This leads to the existence of ground state solutions to \eqref{Phimaineq}  for any speed $ c>0$  when $\gamma>0 $ and any speed $ c>\gamma^2/4 $ when $ \gamma<0 $.   We then prove some convergence results of the constructed profiles either to the KdV or to the BO soliton. In Section \ref{section 3}, we prove the uniqueness and the orbital stability results for ground state solutions to \eqref{Phimaineq} for large positive speeds  without a sign condition on $ \gamma $ and for small positive speeds when $ \gamma>0$.  Finally, in Section \ref{section 4}, we prove decay estimates for any $ H^1(\mathbb{R}) $-solution to \eqref{Phimaineq} as well as for  its first derivatives.

\section{Existence}\label{section 2}
This section is devoted to the proof of the following theorem. 
\begin{theorem} {(Existence and limits of $\varphi_{1, \alpha, \beta}$)}
Let $c=1$. There exists a family of even ground state solutions $\{ \varphi_{1, \alpha, \beta}, 0 < \alpha \le 1  , \beta>-2\sqrt{\alpha}  \}$ of \eqref{Phimaineq2} such that
$$ \lim_{\beta \rightarrow 0} \|\varphi_{1,1, \beta} - Q_{KdV} \|_{H^1} = 0 $$
and 
$$ \lim_{\alpha \rightarrow 0^+} \|\varphi_{1,\alpha, 1} - Q_{BO} \|_{H^{1}} = 0 $$
where $ Q_{KdV} $ and $ Q_{BO}$ are respectively the unique even soliton profiles of speed $ 1$ of the KdV and the BO equation.
\end{theorem}
It is worth noticing that the existence part in Theorem \ref{Thm1} follows directly from the above theorem. Indeed, it is direct to check that $ \varphi $ is a solution to \eqref{Phimaineq2} with $ c=\alpha=1 $ if and only if $ \phi(\cdot)= \lambda^2 \varphi (\lambda\cdot) $ is a solution to \eqref{Phimaineq} with $ c=\lambda^2 $ and $ \gamma = |\lambda| \beta $. Therefore, the above theorem ensures the existence of solutions to \eqref{Phimaineq} for any speed $ c>0 $ in the case $ \gamma>0 $ and for any speed $c>\gamma^2/4 $ in the case $ \gamma<0 $.

\subsection{Preliminary Study}

It is worth noticing that by taking $(\alpha,\beta)=(1,0) $ in \eqref{Phimaineq2c}, we recover the equation of solitons of speed $ c$ of the Korteweg-de Vries equation
\begin{equation}\label{KdV soliton}
c  \varphi - \varphi'' = \frac{1}{2} \varphi^2 
\end{equation}
for which there exists for any $ c>0 $, up to translations, a unique solution (cf. \cite{Berestycki Lions 1983}) of the form
\begin{equation}\label{KdV soliton2}
 Q_{c, KdV}(x)= 3 c \ {\rm sech}^2 (\frac{\sqrt{c}}{2} x )  \; . 
 \end{equation}
Since we will mainly work with the soliton $Q_{1, KdV}(x) =3 \, {\rm sech}^2(\frac{x}{2})$ of speed $c= 1$, we denote it simply by $ Q_{KdV} $ to lighten the notations.

In the same way, by  taking  $(\alpha,\beta)=(0,1) $ in \eqref{Phimaineq2c}, we recover the equation of solitons of speed $ c $ of the Benjamin-Ono equation
\begin{equation}\label{BO soliton}
c \varphi - \mathcal{H} \partial_x \varphi= \frac{1}{2} \varphi^2
\end{equation}
for which there exists for any $c>0 $, up to translations, a unique solution (cf.  \cite{Amick Toland 1991}, \cite{Benjamin 1967}) of the form
\begin{equation}\label{BO soliton2}
Q_{c, BO}(x)= \frac{4c }{1 + c^2 x^2} \; .
\end{equation}
Again to lighten the notations, we denote $\displaystyle Q_{1,B0}(x) =\frac{4}{1+x^2} $ simply by  $ Q_{BO} $.

\begin{lemma}\label{coercive}
Let $ c=1 $, $ \alpha\in ]0,1] $  and $ \beta\in \displaystyle ]-2\sqrt{\alpha},+\infty[ $, then $\sqrt{I_{c,\alpha,\beta}} $ is a norm issued from a scalar product on $ H^1(\mathbb{R})$, and is equivalent to the classical norm  of $ H^1(\mathbb{R}) $. \end{lemma}
\begin{proof}
First, it is clear that $I(u) $ is associated to a bilinear form on $ H^1(\mathbb{R}) $. Let $u \in H^1(\mathbb{R})$. 
Taking the Fourier Transform (Plancherel Identity), it holds 
$$
2 I_{1,\alpha,\beta}(u) = 
\bigg [1+\beta |\xi| +\alpha |\xi|^2 \bigg ] |\hat{u}|^2 \ .$$
In the case $ \beta \ge 0 $, we trivially have 
$$
1+\beta |\xi| +\alpha |\xi|^2 \ge 1+ \alpha |\xi|^2\ge \frac{1}{2} (1+ |\xi|+\alpha |\xi|^2) \ge \frac{1}{2} \min(\alpha,1) (1+ |\xi|+ |\xi|^2) .
$$
Now for $ -2 \sqrt{\alpha}<\beta<0 $, we may rewrite $\beta $ as $  \beta = 
-2 \sqrt{\alpha-\varepsilon} (1-\varepsilon) $ with $ 0<\varepsilon<\alpha  $. 
Noticing that 
$$
\min_{\xi\in \R} \Bigl[ (1-\varepsilon) -2 \sqrt{\alpha-\varepsilon} (1-\varepsilon)  |\xi|+ (\alpha-\varepsilon) |\xi|^2\Bigr]  =0 ,
$$
this leads to 
$$
1+\beta |\xi| +\alpha |\xi|^2 \ge \varepsilon (1+|\xi|^2) \ge \frac{\varepsilon}{2} (1+ |\xi|+ |\xi|^2) .
$$
Hence, it holds 
\begin{equation}\label{coerci}
I_{1,\alpha,\beta} (u) \geq  C_{\alpha,\beta}  \|u\|^2_{H^1}  \; .
\end{equation}
for some $ C_{\alpha,\beta} >0 $.

So, $I$ is coercive in $H^1(\mathbb{R})$. Besides, for both cases of $\alpha$ and $\beta$, one can see that 
$$ I(u) \leq \frac{3}{4} \max(1,|\beta|)  \|u\|^2_{H^1} \ .$$
Eventually, $\sqrt{I_{1,\alpha,\beta}} $ is equivalent to the $H^1$-norm. 
\end{proof}

\begin{lemma}\label{subadd}
$M^\lambda >0$ for any $ \lambda>0 $. Moreover, for any $(\lambda,\lambda_0) \in \mathbb{R}^*_+\times \mathbb{R}^*_+ $ it holds 
\begin{equation} \label{dila}
M^{\lambda} = \bigg( \frac{\lambda}{\lambda_0}\bigg )^\frac{2}{3} M^{\lambda_0} \quad \text{and} \quad M^{\lambda_0}< M^{\lambda} + M^{\lambda_0 - \lambda}    \; .
\end{equation}
\end{lemma}
\begin{proof}
Let $\lambda >0$. Suppose to the contrary that there exists $(u_n)$, a minimizing sequence of $M^\lambda$, such that $I(u_n) \longrightarrow_{n\rightarrow +\infty} 0$. But $I$ is equivalent to $H^1$-norm, then $\|u_n\|_{H^1} \longrightarrow 0$ too. Yet, the constraint $K$ is controlled by the $H^1$-norm as $H^1(\mathbb{R}) \hookrightarrow L^3(\mathbb{R})$. This leads to a contradiction with the fact that the constraint $K(u_n) = \lambda >0$, and thus $M^\lambda >0$. \\
Now, by a dilation argument, one can check that $(u_n)$ is a minimizing sequence of $M^{\lambda}$ if and only if $(v_n)$, defined by $\displaystyle v_n =\bigg (\frac{\lambda_0}{\lambda} \bigg)^\frac{1}{3} u_n$, is a minimizing sequence of $M^{\lambda_0}$.\\
For the second statement in \eqref{dila}, let $\lambda_0 > \lambda >0$. We recall that $f: x \mapsto x^{2/3}$ is a strictly concave function, thus is also strictly subadditive, which principally means that $f(x+y) < f(x) + f(y) $ for $x, \ y >0$. 
This finishes the proof of \eqref{dila}.
\end{proof}
\subsection{Existence of minimizers}

\begin{proposition}\label{minimizer}
Let $ \lambda>0 $. For any $\alpha\in  [0,1]  $ and $ \beta>-2 \sqrt{\alpha} $, there exists a solution to the constrained minimizing problem $M^\lambda_{1,\alpha,\beta} $.

More precisely, set $ \theta=1 $ for  $\alpha>0 $ and $ \theta=1/2 $ for $ \alpha=0$.
Given a minimizing sequence $(\psi_n)$ of $M^\delta_{1,\alpha,\beta} $, there exists a subsequence $(\psi_{n_k})$ converging to  $\psi$ in $H^\theta_e(\mathbb{R})$ that is a solution to the constrained minimization problem $M^\delta_{1,\alpha,\beta} $. 

Finally, $ Q_{KdV} $ and $Q_{BO} $ defined in \eqref{KdV soliton} and \eqref{BO soliton} are respectively the unique solutions to $M_{1,1,0}^{K(Q_{KdV})} $ and $M_{1,0,1}^{K(Q_{BO})} $ .
\end{proposition}

\begin{proof}
We shall treat first the case $\alpha \neq 0$. Let us solve the constraint minimization problem $M^\delta$. We start by taking $(\psi_n)_{n \geq 1}$ a minimizing sequence of the problem $M^\delta$, that is, 
\begin{equation} \label{psin}
I(\psi_n) \longrightarrow_{n \rightarrow +\infty} M^\delta \quad \text{and} \quad K(\psi_n) = \delta>0.
\end{equation}
We claim that there exists a subsequence $(\psi_k)_{k \geq 1}$, and $\psi \in H^1_e$ such that $\psi_k \longrightarrow \psi $ in $H^1_e$ and $\psi$ is a minimizer of $M^\delta$. As $I$ is equivalent to $H^1$-norm according to Lemma \ref{coercive}, then by \eqref{psin} too, $(\psi_n)$ is bounded in $H^1(\mathbb{R})$ which is a Hilbert space. Thus, there exists a subsequence $(\psi_k)$ and a function $\psi$ such that
\begin{equation}\label{weak conv}
\psi_k \rightharpoonup \psi \quad \text{in} \ H^1(\mathbb{R})
\end{equation}
and 
\begin{equation}\label{lim inf}
\|\psi\|^2_{H_1} \leq \liminf_{k \rightarrow + \infty} \|\psi_k\|^2_{H^1}.
\end{equation}
Besides, thanks to the compact embedding $H^1\hookrightarrow L^3_{loc}(\mathbb{R})$, we get
\begin{equation}\label{L3loc conv}
\psi_k \longrightarrow \psi \quad \text{in} \ L^3_{loc}(\mathbb{R})
\end{equation}
and
\begin{equation}\label{ae conv}
\psi_k \longrightarrow \psi \quad \text{a.e \ in} \ \mathbb{R}.
\end{equation}
Now, define the sequence of positive and even functions
$$\rho_n = |\psi_n|^2 + |\partial_x \psi_n|^2; \quad \quad n \in \mathbb{N}^*.$$
As $(\rho_n)$ is bounded in $L^1(\mathbb{R})$, we can extract a subsequence $(\rho_k)$ assuming that 
$$ \lim_{k \rightarrow +\infty} \int_{\mathbb{R}} \rho_k = L <+\infty.$$
By normalizing, we may also assume that 
$$\int_{\mathbb{R}} \rho_k = L \quad \forall \ k \in \mathbb{N}^*. $$
Based on the above analysis, Lemma \ref{coercive} and Lemma \ref{subadd}, we are now eligible to employ the Concentration Compactness Lemma developed by Lions \cite{Lions 1984-1}, \cite{Lions 1984-2}, and we end up with three possibilities:
\begin{itemize}
\item[(P1)] Concentration: There exists $(y_k) \subset \mathbb{R}$ such that for all $\varepsilon >0, \ \exists \ R_\varepsilon >0$; $\forall \ k \in \mathbb{N}^*$ we have
$$ \int_{|x-y_k| \leq R_\varepsilon} \rho_k \ dx \ \geq \ \int_{\mathbb{R}} \rho_k \ dx \quad - \varepsilon , $$
\item[(P2)] Vanishing: For every $R>0$, we have
$$ \lim_{k \rightarrow +\infty} \ \sup_{y \in \mathbb{R}} \int_{|x-y| \leq R} \rho_k \ dx =0,$$
or
\item[(P3)] Dichotomy: There exists $l \in (0,L)$ such that for all $\varepsilon >0, \ \exists \ R>0$, $(R_k)_{k}$ with $R_k \rightarrow +\infty$, $(y_k)_{k \geq 1}$, and $k_0$ such that $\forall \ k \geq k_0$, we have
$$ \bigg| \int_{|x-y_k|\leq R} \rho_k \ dx \quad - l \bigg| < \varepsilon^2, \quad \text{and} \quad \bigg| \int_{R < |x-y_k| < R_k} \rho_k \ dx \bigg| < \varepsilon^2.$$
\end{itemize}
Our plan is to rule out possibilities (P2) and (P3) and so (P1) holds, thus proving that the limit of the minimizing sequence is itself a minimizer of $M^\delta$.\\
To begin with, suppose (P2) is true and take $R=1$. By Sobolev embedding, we have 
\begin{eqnarray}\nonumber
\displaystyle \int_{|x-y| \leq 1} |\psi_k|^3 \ dx & \lesssim & \bigg( \int_{|x-y| \leq 1} |\psi_k|^2 + |\partial_x \psi_k|^2 \ dx \bigg)^\frac{3}{2} \nonumber \\
\displaystyle &=& \bigg ( \int_{|x-y| \leq 1}\rho_k \bigg)^\frac{3}{2}. \nonumber
\end{eqnarray} 
Yet, one can assert from assumption (P2) that for any  $\varepsilon >0$ there exists  $k(\varepsilon)>0$ such that for all $k \geq k(\varepsilon)$ and all $y\in \mathbb{R} $, it holds
$$ \int_{|x-y| \leq 1} \rho_k \ dx \leq \varepsilon.$$
So, for every $y \in \mathbb{R}$,
$$ \int_{|x-y| \leq 1} |\psi_k|^3 \ dx \lesssim \varepsilon^\frac{1}{2} \bigg ( \int_{|x-y| \leq 1}\rho_k \ dx \bigg). $$
Re-summing over intervals centered at even integers, then taking $\varepsilon$ small enough, we get a contradiction to the fact that $K(\psi_k) =\lambda >0$ for all $k \in \mathbb{N}^*$.

Now, suppose (P3) is true. We start by  fixing $ \varepsilon>0 $. The idea is to  decompose $\psi_k$ into two parts in order to construct two minimizing sequences  that  are going infinitely away from each other. This will lead to a contradiction due to the sub-additivity condition following from \eqref{dila}.

Define two smooth cutoff functions $\xi_1$ and $\xi_2$ with supports on $\{ |x| \leq \frac{3}{4}\}$ and $\{ |x| \geq \frac{1}{4} \}$ respectively in such a way that $\xi_1(x)=1$ on $\{ |x| \leq \frac{1}{4} \}$, $\xi_2(x)= 1$ on $\{|x| \geq \frac{3}{4}\}$, and $\xi_1^2 + \xi_2^2 =1$ on $\mathbb{R}$. At this point, construct two sequences 
$$\psi_{k,1}= \xi_1 \bigg ( \frac{|\cdot-y_k|}{R_k} \bigg) \psi_k
\quad \text{ and } \quad \psi_{k,2}= \xi_2 \bigg ( \frac{|\cdot-y_k|}{R_k} \bigg) \psi_k \ ,\ \forall k\geq 1 $$
where $(R_k)_k$ and $y_k$ are consistent with the statement of Dichotomy (P3). The boundedness of  $ (\psi_k) $  in $ H^1(\mathbb{R}) $ ensures that $ (\psi_{k,1})$ and $ (\psi_{k,2}) $ are also bounded in $ H^1(\mathbb{R}) $.
 We claim that these two subsequences satisfy the following property for all $ k \geq k_0=k_0(\varepsilon)$,
\begin{equation}\label{sum K}
K(\psi_k) = K(\psi_{k,1}) + K(\psi_{k,2}) + O(\varepsilon).
\end{equation}
Indeed, regarding \eqref{sum K}, we have
\begin{eqnarray}\nonumber
\displaystyle \int (\psi_{k,1})^3 + \int (\psi_{k,2})^3 &=& \int \xi_{1,R_k}^2 \ \psi_k^3 + \int \xi_{2,R_k}^2 \ \psi_k^3 \nonumber \\
\displaystyle &+& \int (\xi_{1,R_k}^3 - \xi^2_{1,R_k}) \ \psi_k^3 +  \int (\xi_{2,R_k}^3 - \xi^2_{2,R_k}) \ \psi_k^3 \ . \label{sumkform} 
\end{eqnarray} 
For $k$ large enough, we have 
$$\xi_{1,R_k} = 1 \ \text{on} \ [- \frac{R_k}{4}, \frac{R_k}{4} ] \quad \text{and} \quad \xi_{1,R_k} = 0 \ \text{on} \ \mathbb{R} \setminus [-\frac{3 \ R_k}{4}, \frac{3 \ R_k}{4}] $$
so supp $(\xi_{1,R_k}^3 - \xi^2_{1,R_k}) \subset D=[-\frac{3 R_k}{4}, -\frac{R_k}{4}] \cup [\frac{R_k}{4}, \frac{3 R_k}{4}]$. Hence, by Cauchy-Schwarz inequality, the Sobolev embedding $H^\frac{1}{2} \hookrightarrow L^4(\mathbb{R})$, and the Dichotomy assumption, we get 
\begin{eqnarray}\nonumber
\displaystyle  \int_{\mathbb{R}} (\xi_{1,R_k}^3 - \xi^2_{1,R_k}) \ \psi_k^3 &\leq & \int_{D} |\psi_k^3| = \int_{D} |\psi_k| \ |\psi_k|^2 \nonumber \\
\displaystyle &\lesssim & \bigg( \int_{D} |\psi_k|^2 \bigg)^\frac{1}{2} \ \bigg( \int_{D} |\psi_k|^4 \bigg)^\frac{1}{2} \nonumber \\
\displaystyle & \lesssim & \bigg( \int_{\frac{R_k}{4} <|x-y_k| < \frac{3 R_k}{4}} \rho_k \bigg)^\frac{1}{2} \|\psi_k\|_{L^4}^2 \nonumber \\
\displaystyle & \lesssim & \varepsilon \ \|\psi_k\|_{H^\frac{1}{2}}^2 \ .\nonumber
\end{eqnarray}
Repeating this same work for the term $\displaystyle \int_{\mathbb{R}} (\xi_{2,R_k}^3 - \xi^2_{2,R_k}) \ \psi_k^3$, then substituting in \eqref{sumkform}, we obtain \eqref{sum K}.
\\
We also claim that the two subsequences $(\psi_{k,1})$ and $(\psi_{k,2})$ satisfy 
\begin{equation}\label{sum I}
I(\psi_k) = I(\psi_{k,1}) + I(\psi_{k,2}) + O(\varepsilon) 
\end{equation} 
for all $ k \geq k_0$ where we  choose $ k_0=k_0(\varepsilon)$ such that $ 1/R_{k_0} <\varepsilon $. To prove the latter, we follow \cite{Albert 1999}. It holds 
\begin{eqnarray}\nonumber
\displaystyle I(\psi_{k,1}) + I(\psi_{k,2}) &=& \frac{c}{2} \bigg[ \int (\psi_{k,1})^2 + (\psi_{k,2})^2 \bigg] + \frac{\alpha}{2} \bigg[ \int (\psi_{k,1}')^2 + (\psi_{k,2}')^2 \bigg]\nonumber \\
\displaystyle & +& \frac{\beta}{2} \bigg[\int (D_x^{\frac{1}{2}} \psi_{k,1})^2 + (D_x^{\frac{1}{2}} \psi_{k,2})^2 \bigg] \nonumber \\
\displaystyle &:=& L_{c,\alpha} + L_{\beta} \nonumber
\end{eqnarray}
where, for more clarity, we treat the first and second lines separately. Let $\xi_{i,R}(\cdot)$ denote $\xi_i(\frac{\cdot}{R})$, we can see that
\begin{eqnarray}\nonumber
\displaystyle L_{c,\alpha} &=& \frac{c}{2} \bigg[ \int (\xi_{1,R_k} \ \psi_k )^2 + (\xi_{2,R_k} \ \psi_k )^2 \bigg] \nonumber \\
\displaystyle &+& \frac{\alpha}{2} \bigg[ \int (\xi_{1,R_k}' \ \psi_k)^2 + 2 \int \xi_{1,R_k} \ \xi'_{1,R_k} \ \psi'_k \ \psi_k \ +\int (\xi_{1,R_k} \ \psi'_k)^2 \bigg] \nonumber \\
\displaystyle &+& \frac{\alpha}{2} \bigg[ \int (\xi_{2,R_k}' \ \psi_k)^2 + 2 \int \xi_{2,R_k} \ \xi'_{2,R_k} \ \psi'_k \ \psi_k \ +\int (\xi_{2,R_k} \ \psi'_k)^2 \bigg] \nonumber .
\end{eqnarray}
Using the facts that 
\begin{equation} \label{props}
\xi_{1,R_k}^2 + \xi_{2,R_k}^2 =1 \quad \text{and} \quad \displaystyle \|\xi_{i,R_k}'\|_{L^\infty} \leq \frac{1}{R_k} \ \|\xi_i\|_{L^\infty} \lesssim \frac{1}{R_k}\lesssim \varepsilon,
\end{equation}
we get
\begin{equation}\label{L1}
L_{c,\alpha} = \frac{c}{2} \int (\psi_k )^2 + \frac{\alpha}{2} \int (\psi_k')^2 + O(\varepsilon). 
\end{equation}
Now to control the contribution of  $L_{\beta}$, we recall that 
$$ \int (D_x^\frac{1}{2} f)^2 = \int f' \mathcal{H} f \quad \forall f\in H^1(\mathbb{R}) $$ so that $L_{\beta}$ may be rewritten as
$$ L_{\beta} = \frac{\beta}{2} \bigg[ \int (\psi_{k,1})' \mathcal{H} \psi_{k,1} + \int (\psi_{k,2})' \mathcal{H} \psi_{k,2} \bigg] .$$
We need the following commutator  Lemma; check \cite{Albert 1999} and \cite{Albert Bona Saut 1961}.
\begin{lemma}\label{commutator}
There exists a constant $A >0$ such that if $\theta$ is any continuously differentiable function with $ \theta' \in L^\infty$ then
$$ \|[\mathcal{H} \partial_x, \theta] \ f\|_{L^2} \leq A \ \|\theta'\|_{L^\infty} \ \|f\|_{L^2}, \quad \forall f\in L^2(\mathbb{R}) \;,  $$
where $[\mathcal{H} \partial_x,\theta] f = \mathcal{H} \partial_x(\theta f) - \theta \ \mathcal{H} \partial_x(f)$. 
\end{lemma}
Looking at the first component of $L_{\beta}$ (since the other component can be treated identically), we can see that
\begin{eqnarray}\nonumber
\displaystyle \int (\xi_{1,R_k} \ \psi_k) \ [\mathcal{H} \partial_x, \xi_{1, R_k}] \psi_k &=& \int \xi_{1,R_k} \ \psi_k \ \mathcal{H} \partial_x (\xi_{1, R_k} \ \psi_k) - \int \xi^2_{1,R_k} \ \psi_k \ \mathcal{H} \partial_x \psi_k \nonumber \\
\displaystyle &=& - \int \mathcal{H}(\xi_{1,R_k} \ \psi_k) \ (\xi_{1,R_k} \ \psi_k)' + \int (\xi_{1,R_k}^2 \ \psi_k)' \ \mathcal{H} \psi_k \nonumber \\
\displaystyle &=&- \int (\xi_{1,R_k} \ \psi_k)' \mathcal{H}(\xi_{1,R_k} \ \psi_k) + 2 \int \xi_{1,R_k} \ \xi_{1,R_k}'  \ \psi_k \ \mathcal{H} \psi_k \nonumber \\
\displaystyle &+& \int \xi_{1,R_k}^2 \ \psi_k' \ \mathcal{H} \psi_k \ .\nonumber
\end{eqnarray}
Hence,
\begin{eqnarray}
L_{\beta}= \frac{\beta}{2} \bigg[ &-&\int (\xi_{1,R_k} \ \psi_k) \ [\mathcal{H} \partial_x, \xi_{1, R_k}] \psi_k - \int (\xi_{2,R_k} \ \psi_k) \ [\mathcal{H} \partial_x, \xi_{2, R_k}] \psi_k \nonumber \\
\displaystyle &+& \quad 2 \int \xi'_{1,R_k} \ \xi_{1,R_k} \ \psi_k \mathcal{H}(\psi_k) + 2 \int \xi'_{2,R_k} \ \xi_{2,R_k} \ \psi_k \mathcal{H}(\psi_k) \nonumber \\
\displaystyle &+& \quad \int \xi_{1,R_k}^2 \ \psi'_k \mathcal{H} (\psi_k) + \int \xi_{2,R_k}^2 \ \psi'_k \mathcal{H} (\psi_k) \bigg] .\nonumber
\end{eqnarray}
Using  \eqref{props} and H\"older's inequality, we can check that for $i=1, 2$,
\begin{eqnarray}\nonumber
\Bigl| \int \xi'_{i,R_k} \ \xi_{i,R_k} \ \psi_k \mathcal{H}(\psi_k)\Bigr| &\leq & \frac{1}{R_k} \int_D | \psi_k \ \mathcal{H} \psi_k |
 \leq \frac{1}{R_k} \ \|\psi_k\|_{L^2}^2 \lesssim \frac{1}{R_k} \ . \label{LB2}
\end{eqnarray}
In the same way, using also Lemma \ref{commutator}, we can check that for $i=1, 2$,
\begin{eqnarray}\nonumber
\displaystyle\Bigl|  \int (\xi_{i,R_k} \ \psi_k) \ [\mathcal{H} \partial_x, \xi_{i, R_k}] \psi_k\Bigr|  &\leq & \|\xi_{i,R_k}\|_{L^\infty} \ \|\psi_k\|_{L^2} \ \|\xi'_{i,R_k}\|_{L^\infty} \ \|\psi_k\|_{L^2} \nonumber \\
\displaystyle &\lesssim  & \frac{1}{R_k} \; .\label{LB1}
\end{eqnarray}
 Therefore, recalling that $
  \xi^2_{1,R_k} + \xi^2_{2,R_k} =1 $ we end up with 
  \begin{equation}\label{L2}
L_{\beta} = \frac{\beta}{2} \int \psi'_k \mathcal{H} \psi_k + O(\frac{1}{R_k})  \ .
\end{equation}
Eventually, combining \eqref{L1} and \eqref{L2}, we obtain \eqref{sum I}. 
We can return back to our assumption. By Sobolev embedding, $(K(\psi_{k,i}))$ are  bounded for $i=1,2$, and we may pass to a subsequence to define $\delta_i(\varepsilon)= \lim_{k \rightarrow +\infty} K(\psi_{k,i})$. Since it holds for any $ \varepsilon>0 $ and $\delta_i(\varepsilon)$ are uniformly bounded, we can choose a sequence $\varepsilon_j \rightarrow 0$ such that the limits $\delta_i = \lim_{j \rightarrow +\infty} \delta_i(\varepsilon_j)$ exist for $i=1,2$, and we set $\delta= \delta_1 + \delta_2$. We differentiate three cases:
\begin{itemize}
\item[1st case:] If $0< \delta_1 < \delta$, then $\delta_2 >0$. It follows from \eqref{sum I} that
$$
\displaystyle I(\psi_k) =  I(\psi_{k,1}^j) + I(\psi_{k,2}^j) + O(\varepsilon_j) 
$$
where $ (\psi_{k,i}^j) $, $i=1,2$,  is  the sequence $ (\psi_{k,i}) $ associated with $ \varepsilon_j $, $ j\in \N $. 
Passing to the limit as $ k\to +\infty $, using \eqref{dila} and that $(\psi_k)$ is a minimizing sequence of $M^\delta$,  it follows that 
\begin{eqnarray}
M^\delta  & \geq & M^{\delta_1(\varepsilon_j)} + M^{\delta_2(\varepsilon_j)} + O(\varepsilon_j) \nonumber \\
\displaystyle &\geq &  \bigg[ \bigg(\frac{\delta_1(\varepsilon_j)}{\delta} \bigg)^\frac{2}{3} + \bigg(\frac{\delta_2(\varepsilon_j)}{\delta} \bigg)^\frac{2}{3} \bigg ] M^\delta + O(\varepsilon_j). \nonumber
\end{eqnarray}
Taking  $j \rightarrow +\infty$, we get the contradiction
$$ M^\delta \geq \bigg[ \bigg(\frac{\delta_1}{\delta} \bigg)^\frac{2}{3} + \bigg(\frac{\delta_2}{\delta} \bigg)^\frac{2}{3} \bigg ] M^\delta > M^\delta. $$

\item[2nd case:] \ If $\delta_1 =0$ or $\delta_1 = \delta$, we shall only treat the case $ \delta_1=0 $ since the other one  is similar. Recall that $ \delta_1=0 $ forces $ \delta_2=\delta $. We can check that
\begin{eqnarray}\nonumber
\displaystyle I(\psi_k) &=&  I(\psi_{k,1}^j) + I(\psi_{k,2}^j) + O(\varepsilon_j)\nonumber\\
\displaystyle & \geq & \int_{|x-y| \leq R_k(\varepsilon_j)} |\psi_k|^2 + |\partial_x \psi_k|^2 \ dx \ + \ M^{K(\psi_{k,2}^j)}  + O(\varepsilon_j) \nonumber \\
\displaystyle & \geq & l + \bigg( \frac{K(\psi_{k,2}^j)}{\delta} \bigg)^\frac{2}{3} M^\delta + O(\varepsilon_j) \nonumber
\end{eqnarray}
using \eqref{sum I}, \eqref{dila} and assumption (P3). Letting $k\to +\infty  $ and then $ j \rightarrow +\infty$, and recalling that $l >0$, we get
$$ M^\delta \geq l + M^\delta > M^\delta$$
which is a contradiction.
\item[3rd case:] \ If $\delta_1 > \delta$ or $\delta_1 <0$; in the former, it follows from \eqref{sum I} that 
$$ I(\psi_k) \geq I(\psi_{k,1}^j) + O(\varepsilon_j) \geq \bigg( \frac{K(\psi_{k,1}^j)}{\delta} \bigg)^\frac{2}{3} M^\delta + O(\varepsilon_j).$$
Letting $k\to +\infty  $ and then $ j \rightarrow +\infty$ we get the following contradiction
$$ M^\delta \geq \bigg(\frac{\delta_1}{\delta} \bigg)^\frac{2}{3} M^\delta > M^\delta \quad \text{as} \ \delta_1 > \delta .$$
Finally, the latter can be treated in a similar way since $\delta_1 <0$ forces  $ \delta_2>\delta $.
\end{itemize}
Eventually, (P1) holds. Yet, as we are dealing with even functions, we notice that in assumption (P1), we shall assume that $y_k=0$ since a non-bounded $(y_k)$ will lead to a contradiction. Indeed, taking $ \varepsilon =L/4 $ and  $ |y_k|>R_{L/4} $,  (P1) would lead to 
$$
L=\int_{\R} \rho_k(x) \, dx \ge 2 \int_{|x-y_k|\le R_{L/4} } \rho_k(x) \, dx \ge \frac{3}{2} L \; .
$$
\\
Now, since $ (\psi_k) $ is bounded in $ H^1(\mathbb{R}) \hookrightarrow L^\infty(\mathbb{R})$, it follows from ($P_1$) that for any $ \varepsilon>0 $  there exists $ R_\varepsilon >0 $ with $ R_\varepsilon\to +\infty $ as $ \varepsilon \to 0 $ such that 
$$ \int_{\mathbb{R} \setminus [-R_\varepsilon, R_\varepsilon]} |\psi_k|^3 \ dx \leq \|\psi_k\|_{L^\infty} 
 \int_{\mathbb{R} \setminus [-R_\varepsilon, R_\varepsilon]} |\psi_k|^2 \ dx 
 \le C \, \int_{\mathbb{R} \setminus [-R_\varepsilon, R_\varepsilon]} \rho_k \, dx \le \varepsilon, \quad \quad \forall k\in \N^* .$$
Since  $ \psi \in H^1(\R) \hookrightarrow L^3(\R) $, we deduce that for any $ k\in \N^* $ and $ \varepsilon>0 $,  
$$ \int_{\mathbb{R} \setminus [-R_\varepsilon, R_\varepsilon]} |\psi_k-\psi|^3 \ dx \lesssim
\Bigl(  \int_{\mathbb{R} \setminus [-R_\varepsilon, R_\varepsilon]} |\psi_k|^3 \ dx +  \int_{\mathbb{R} \setminus [-R_\varepsilon, R_\varepsilon]} |\psi|^3 \ dx\Bigr) \lesssim \varepsilon + \alpha(\varepsilon) 
$$
with $ \alpha(y) \to 0 $ as $ y\to 0 $. Combining this last inequality with \eqref{L3loc conv}, we infer that 
\begin{equation}\label{L3 conv}
\psi_k \longrightarrow \psi \quad \quad \text{in} \ L^3(\mathbb{R}).
\end{equation}
Now, according to Lemma \ref{coercive}, $I$ is lower semi-continuous with respect to the weak $H^1$-topology so that
$$ I(\psi) \leq \lim_{k \rightarrow + \infty} I(\psi_k) = M^\delta.$$
Also, owing to \eqref{L3 conv}, we get $K(\psi)= \delta$. Hence, $I(\psi)= M^\delta= \displaystyle \lim_{k\to +\infty} I(\psi_k) $ and $\psi$ is a minimizer of $M^\delta$. Moreover, this last equality together with the weak $ H^1$ convergence \eqref{weak conv} and again Lemma  \ref{coercive} ensure that
\begin{equation}\label{H1 conv}
\psi_k \longrightarrow \psi \quad \text{in} \  H^1_e.
\end{equation}

We now consider the case $\alpha=0$ that corresponds to the Benjamin-Ono case.
In this case, we can perform a similar analysis with some slight modifications since this time we work in $ H^{1/2}(\mathbb{R}) $ and thus take  $\rho_k$ to be 
$$\rho_k := |\psi_k|^2 + |D_x^\frac{1}{2} \psi_k|^2.$$
The main difference with respect to the case $ \alpha>0 $ is the way to disclaim the vanishing case due to the nonlocal property of the $ H^{1/2}$-norm. Let us explain how we can proceed. We introduce a smooth non negative function $ \eta\in C^\infty_c([-2,2]) $ such that $ \eta \equiv 1 $ on $[-1,1] $ and $ \sum_{m\in \mathbb{Z}} \eta(\cdot-m)^2\equiv 1 $ on $\mathbb{R}$. Denoting $ \eta(\cdot-m) $ by $\eta_k$, and using the fact that $\eta_m \eta_{m'} \equiv 0$ on $\mathbb{R}$  as soon as $ |m-m'|\ge 2$, the following chain of inequalities holds :
$$
\bigg |\int_{\mathbb{R}}  \psi_k^3 \ \bigg | \le \int_{\mathbb{R}}  |\psi_k|^3=\int_{\mathbb{R}} |\sum_{m \in \mathbb{Z}} \eta_m \psi_k|^3 \lesssim \sum_{m\in \mathbb{Z}} |\eta_m \psi_k|^3\; .
$$
Therefore, by the Sobolev embedding $ H^{1/6}(\mathbb{R}) \hookrightarrow L^3(\mathbb{R}) $, we get
\begin{eqnarray*}
\bigg |\int_{\mathbb{R}}  \psi_k^3 \ \bigg | & \lesssim &  \sum_{k\in\mathbb{Z}} \|\eta_m \psi_k\|_{L^2}^{5/2} \  \|\eta_m \psi_k\|_{H^{1/2}}^{1/2} \\
& \lesssim & \sup_{m \in \mathbb{Z}}   \|\eta_m \psi_k\|_{L^2}^{1/2}  \|\eta_m \psi_k\|_{H^{1/2}}^{1/2} \sum_{m\in\mathbb{Z}} \int_{\mathbb{R}} \eta_m^2 \psi_k^2 \\
&  \lesssim & \sup_{m\in\mathbb{Z}}  \|\eta_m \psi_k\|_{L^2}^{1/2}  \|\eta_m \psi_k\|_{H^{1/2}}^{1/2} \|\psi_k\|_{L^2}^2 \ .
\end{eqnarray*}
Since, on account of assumption ($P_2$),
$$
\sup_{m\in\mathbb{Z}} \|\eta_m \psi_k\|_{L^2}^2 \le \sup_{m\in\mathbb{Z}} \int_{m-2}^{m+2} \psi_k^2 \longrightarrow  0 \quad \text{as}\;  k\to \infty
$$
  and by Sobolev inequality, $  \|\eta_m \psi_k\|_{H^{1/2}}\lesssim \|\psi_k\|_{H^{1/2}}\lesssim 1$, $ \forall m \in \mathbb{Z}$,  this forces $\displaystyle \int_{\mathbb{R}}  \psi_k^3\to 0 $ which hence contradicts $  K(\psi_k)=\lambda>0 $ for all $k\in \mathbb{N}$.
  
Finally, it remains to prove the uniqueness of the solution to $ M^{K(Q_{KdV})}_{1,1,0} $ and $M^{K(Q_{BO})}_{1,0,1}$. We notice that, in the case $\alpha>0 $, a solution $ \psi $ to $M^\lambda_{1,\alpha,\beta} $ has to satisfy the associated Euler-Lagrange functional equation
\begin{equation}\label{EL}
\langle -\alpha \psi''-\beta \mathcal{H} \psi'+\psi -\frac{\mu}{2} \psi^2,  f \rangle_{H^{-1}, H^{1}} =0 \ , \quad \quad \forall f\in H^1_e(\mathbb{R}), 
\end{equation}
for some $\mu \in \mathbb{R}$ called the Lagrange multiplier. This is justified as both $\psi \mapsto I_{1,\alpha,\beta}(\psi)$ and $\psi \mapsto K(\psi)$ are $C^1$-functionals from $H^1(\mathbb{R})$ to $\mathbb{R}$ when $ \alpha>0 $ and from $H^{1/2}(\mathbb{R}) $ into $\mathbb{R}$ when $ \alpha=0 $. We remark that due to the even parity of the left-hand side member of the $ H^{-1},H^1$-duality product in \eqref{EL}, the latter is also orthogonal to any odd test function in $H^1(\mathbb{R})$ and thus satisfies 
\begin{equation}\label{eqEL2}
-\alpha \psi''-\beta \mathcal{H} \psi'+\psi =\frac{\mu}{2} \psi^2
\end{equation}
at least in the distributional sense. In view of Lemma \ref{coercive}, a standard bootstrapping argument ensures that $ \psi \in H^\infty(\mathbb{R}) $ and thus is actually a strong solution to \eqref{eqEL2}. In the case $ (\alpha,\beta)=(1,0)$, as $ Q_{KdV} $ is the unique even $ H^1(\mathbb{R}) $ solution to \eqref{KdV soliton} (cf. \cite{Berestycki Lions 1983}), it forces $ \varphi = C \, Q_{KdV} $ for some $ C\in \R $, but then the constraint $ K(\varphi)=K(Q_{KdV}) $ ensures that $ C=1 $ and thus $ \psi=Q_{KdV} $. This proves that $ Q_{KdV} $ is the unique solution to $ M^{K(Q_{KdV})}_{1,1,0}$. \\
The proof that $Q_{BO}$  is the unique solution to $M^{K(Q_{BO})}_{1,0,1}$ is exactly the same, only this time using that $\psi \mapsto I_{1,0,1}(\psi)$ and $\psi \mapsto K(\psi)$ are $C^1$-functionals from $H^{1/2}(\mathbb{R})$ to $\mathbb{R}$ and that $Q_{BO} $ is the unique solution to \eqref{BO soliton} (cf. \cite{Amick Toland 1991}, \cite{Benjamin 1967}).
\end{proof}

\begin{corollary} \label{gss}
Let $c=1$. For any $\alpha \in ]0,1] $ and any $\beta >-2\sqrt{\alpha} $, there exists an even  ground state solution $\varphi$ to \eqref{Phimaineq2}. 
\end{corollary}
\begin{proof} We just  proved above that any solution $ \varphi $ to $ M^\lambda_{1,\alpha,\beta} $, with $\lambda>0 $,  satisfies \eqref{eqEL2}  for some $ \mu\in\mathbb{R} $ and belongs to $ H^\infty(\mathbb{R}) $. 
Multiplying \eqref{eqEL2} by $\varphi$, integrating over $\mathbb{R}$, we get
$$\int_\mathbb{R} \varphi^2 \ + \ \beta \int_\mathbb{R} (D_x^{\frac{1}{2}} \varphi)^2 \ + \alpha \int_\mathbb{R} (\varphi')^2 \ = 
\frac{\mu}{2} \int_\mathbb{R} \varphi^3  ,$$
that is,
\begin{equation}\label{oi}
2 \ I_{1,\alpha,\beta}(\varphi) = 3 \ \mu \ K(\varphi), \quad \text{or \ equivalently}, \quad \mu =\frac{2}{3} \ \frac{I_{1,\alpha,\beta}(\varphi)}{K(\varphi)} =\frac{2}{3}  \frac{I_{1,\alpha,\beta}(\varphi)}{\lambda}>0 \; .
\end{equation}
Hence, one can check that 
\begin{equation}
\psi_{1,\alpha,\beta}= \frac{\mu}{2} \ \varphi
\end{equation} 
is a non trivial solution to  \eqref{Phimaineq2} and to $ M^{\lambda (\mu/2)^3}_{1,\alpha,\beta} $. Therefore it is a ground state solution to \eqref{Phimaineq2}.

\end{proof}

\subsection{Limits}
At this point, we tend to prove that in the case $c=1$, the minimizers of $M^{K(Q_{KdV})}_{1,1,\beta_n}$ converge to the KdV soliton as $\beta_n$ converges to zero. Similarly, the minimzers of $M^{K(Q_{BO})}_{1,\alpha_n, 1}$ converge to the BO soliton as $\alpha_n$ converges to zero.

\subsubsection{Convergence Case 1: $\{\alpha =1, \ |\beta| \ll 1 \}$}

\begin{proposition} \label{Conv1}
Let $\beta_n \rightarrow 0$ and associate to it $(\psi_{1,1,\beta_n})_{n \geq 1}$ a sequence of minimizer solutions to the problem $ \displaystyle M^{K(Q_{KdV})}_{1,1,\beta_n}$. Then, $ (\psi_{1,1,\beta_n})_{n \geq 1} $ converges to the KdV soliton profile $Q_{KdV}$ in $H^1(\mathbb{R})$.
\end{proposition}
\begin{proof}
Let $\alpha=1$, $(\beta_n)_{n \geq 1} \subset \mathbb{R}^\ast $ such that $\displaystyle \lim_{n \rightarrow + \infty} \beta_n =0$, and let $(\psi_{1,1,\beta_n})$ be a sequence of minimizers of $M^{K(Q_{KdV})}_{1,1,\beta_n}$. According to Lemma \ref{coercive}, $I_{1,1,\beta_n}$ is equivalent to the $H^1$-norm uniformly for $ \beta\in [0,1/2] $, hence we get that $(\psi_{1,1,\beta_n})_{n\ge 1}$ is bounded in $H^1(\mathbb{R})$. Besides, we have
\begin{eqnarray}
\displaystyle M^{K(Q_{KdV})}_{1,1,0} &=& I_{1,1,\beta_n}(Q_{KdV}) -\frac{\beta_n}{2} \int_{\mathbb{R}} (D_x^{\frac{1}{2}}(Q_{KdV}))^2 \nonumber \\
\displaystyle &\geq & M^{K(Q_{KdV})}_{1,1,\beta_n} - \frac{|\beta_n|}{2} \int_{\mathbb{R}} (D_x^{\frac{1}{2}}(Q_{KdV}))^2 \nonumber 
\end{eqnarray}
while
\begin{eqnarray}\nonumber
\displaystyle  M^{K(Q_{KdV})}_{1,1,0} &\leq & I_{1,1,0}(\psi_{1,1,\beta_n}) =  I_{1,1,\beta_n}(\psi_{1,1,\beta_n}) - \frac{\beta_n}{2} \int_{\mathbb{R}} (D_x^{\frac{1}{2}}(\psi_{1,1,\beta_n}))^2 \nonumber \\
\displaystyle &=& M^{K(Q_{KdV})}_{1,1,\beta_n} - \frac{\beta_n}{2} \int_{\mathbb{R}} (D_x^{\frac{1}{2}}(\psi_{1,1,\beta_n}))^2 . \nonumber
\end{eqnarray}
Taking $\beta_n \rightarrow 0$, we get
\begin{equation}\label{pbconv1}
M^{K(Q_{KdV})}_{1,1,\beta_n} \longrightarrow M^{K(Q_{KdV})}_{1,1,0} \ .
\end{equation}
Hence, $\psi_{1,1,\beta_n}$ is a minimizing sequence of $M_{1,1,0}^{K(Q_{KdV})}$. By Proposition \ref{minimizer}, there exists a subsequence $(\psi_{1,1,\beta_{n_k}})$ that converges in $H^1(\mathbb{R})$ to a solution $\psi$ of $M_{1,1,0}^{K(Q_{KdV})}$. Recalling that according to Proposition \ref{minimizer}, $ Q_{KdV} $ is the unique solution to $M_{1,1,0}^{K(Q_{KdV})}$, we deduce that $(\psi_{1,1,\beta_{n_k}})$ converges to $Q_{KdV}$ in $H^1(\mathbb{R})$.
Finally, by the uniqueness of the possible limits in $H^1(\mathbb{R})$ of its subsequences, we deduce that $ \psi_{1,1,\beta_n} \rightarrow Q_{KdV} $  in $ H^1(\mathbb{R}) $. 
\end{proof}
\subsubsection{Convergence Case 2: $\{\alpha \ll 1, \ \beta = 1 \}$}
\begin{proposition} \label{pconv2}
Let $\alpha_n \rightarrow 0$, and $(\psi_{1,\alpha_n,1})$ be a sequence of minimizers to the constraint problem $M^{K(Q_{BO})}_{1,\alpha_n, 1}$. Then, $ (\psi_{1,\alpha_n,1})_{n \geq 1} $ converges to the BO soliton profile $Q_{BO}$ in $H^1(\mathbb{R})$.

\end{proposition}

\begin{proof}
Let $\beta=1$ and $(\alpha_n)_{n \geq 1} \subset \mathbb{R}^\ast $ be any sequence such that $\displaystyle \lim_{n \rightarrow + \infty} \alpha_n =0$, and let $(\psi_{\alpha_n})$ be an associated sequence of minimizers of $M^{K(Q_{BO})}_{1,\alpha_n,1}$ with $ (\mu_n) $ the associated sequence of Lagrange multiplier in \eqref{eqEL2}. Contrarily to the preceding case where $\alpha$ was converging to zero, we are only left with boundedness of $(\psi_{1,\alpha_n,1})$ in $ H^{1/2}_e(\mathbb{R}) $ since $ I_{1,\alpha_n,1}(\psi_{1,\alpha_n,1}) \ge \|\psi_{1,\alpha_n,1}\|_{H^{1/2}} $. To get the boundedness in $ H^1(\mathbb{R}) $, we have to work a little more.
First, we notice that $\mu_n$ is bounded since 
\begin{equation}\label{mu2}
0\le \mu_n = \frac{2}{3} \ \frac{I_{1,\alpha_n,1}(\psi_{1,\alpha_n,1})}{K(Q_{BO})} = \frac{2}{3} \ \frac{M^{K(Q_{BO})}_{1,\alpha_n,1}}{K(Q_{BO})} \ \leq \ \frac{2}{3} \ \frac{I_{1,\alpha_n,1}(Q_{BO})}{K(Q_{BO})}\;  . 
\end{equation}
Second, applying the Fourier Transform to \eqref{eqEL2}, we get
\begin{equation}\label{psialpha}
 \bigg [ 1+ |\xi| + \alpha_n |\xi|^2 \bigg ] \widehat{\psi_{1,\alpha_n,1}} (\xi) =  \mu_n  \, \widehat{\psi^2_{1,\alpha_n,1}} \quad \text{for \ any} \ \xi \in \mathbb{R}.
\end{equation}
By Plancherel equality and Sobolev inequalities, we have 
$$ \|\widehat{\psi^2_{1,\alpha_n,1}}\|_{L^2} = \|\psi_{1,\alpha_n,1}\|^2_{L^4} \lesssim \|\psi_{1,\alpha_n,1}\|^2_{H^{\frac{1}{2}}} \ .$$
 So, taking the $ L^2$-norm of both sides of \eqref{psialpha} we obtain that $ (\psi_{1,\alpha_n,1}) $ is actually bounded in $H^1(\mathbb{R}) $. Note that, multiplying  \eqref{psialpha} by $ \sqrt{1+\xi^2} $, taking again the $ L^2$-norm of both sides of \eqref{psialpha} and using that $H^1(\R)$ is an algebra, we even obtain that $ (\psi_{1,\alpha_n,1}) $ is bounded in $H^2(\mathbb{R})$.
\\
Now, we notice  that 
\begin{eqnarray}\nonumber
\displaystyle M^{K(Q_{BO)}}_{1,0,1} &=& I_{1,\alpha_n,1}(Q_{BO}) - \frac{\alpha_n}{2} \int (\partial_x Q_{BO})^2 \geq M^{K(Q_{BO})}_{1,\alpha_n,1} - \frac{\alpha_n}{2} \int (\partial_x Q_{BO})^2 \nonumber
\end{eqnarray}
while, on the other hand,
\begin{eqnarray}\nonumber
\displaystyle M^{K(Q_{BO})}_{1,0,1} & \leq & I_{1,0,1}(\psi_{1,\alpha_n,1}) = I_{1,\alpha_n,1}(\psi_{1,\alpha_n,1}) - \frac{\alpha_n}{2} \int (\partial_x \psi_{1,\alpha_n,1})^2 \nonumber \\
\displaystyle &=&  M^{K(Q_{BO})}_{1,\alpha_n,1} - \frac{\alpha_n}{2} \int (\partial_x \psi_{1,\alpha_n,1})^2 .
\end{eqnarray}
Thanks to the boundedness of $(\psi_{1,\alpha_n,1})$ in $H^1$, we may pass to the limit on the above two inequalities and obtain the convergence
\begin{equation}\label{prbconv}
M^{K(Q_{BO})}_{1,\alpha_n,1} \longrightarrow M^{K(Q_{BO})}_{1,0,1} \quad \text{as} \ \alpha_n \rightarrow 0 \ .
\end{equation}
Hence, $(\psi_{1,\alpha_n,1})$ is a minimizing sequence for $M^{K(Q_{BO})}_{1,0,1}$. Proposition \ref{minimizer} with $\alpha=0$ then ensures that there exists a subsequence $ (\psi_{1,\alpha_{n_k},1}) $  which converges in $ H^{1/2}_e(\mathbb{R}) $ to a minimizer $ \psi $ of $M^{K(Q_{BO})}_{1,0,1}$ and the uniqueness of the solution to $M^{K(Q_{BO})}_{1,0,1}$ ensures, as in the proof of Proposition \ref{Conv1}, that $(\psi_{1,\alpha_n,1}) $ converges to $Q_{BO} $ in  $ H^{1/2}(\mathbb{R}) $. Finally, the boundedness of  $(\psi_{1,\alpha_n,1})$ in $ H^2(\mathbb{R}) $ forces this convergence to hold also in $ H^1(\mathbb{R}) $.
\end{proof}
\section{Uniqueness and Orbital Stability of the ground state solutions}\label{section 3}
In this section, we will first prove the uniqueness of the even ground state solutions of \eqref{Phimaineq2} and then their orbital stability over the two cases of study.
\begin{proposition} \label{uniqueness}
 There exists $\tilde{\beta}_0 >0$ (respectively $\tilde{\alpha}_0>0$) such that $\forall \ |\beta| < \tilde{\beta}_0$ (respectively $\forall \ 0 \leq \alpha <\tilde{\alpha}_0$)  there exists a unique even ground state solution $Q_{1,1,\beta}$ to \eqref{Phimaineq2} with $\alpha=1$ (respectively $Q_{1,\alpha, 1}$ to \eqref{Phimaineq2} with $\beta =1$). Besides, the maps
$$ ]-\tilde{\beta}_0, \tilde{\beta}_0[ \ni \beta \longmapsto Q_{1,1,\beta} \in H^1(\mathbb{R}) $$
and 
$$  [0, \tilde{\alpha}_0[ \ni \alpha \longmapsto Q_{1,\alpha,1} \in H^1(\mathbb{R}) $$
are continuous.
\end{proposition}

\begin{proof}\text{ } \\
\begin{center}
\underline{\bf{Case 1}:} $\{\alpha = 1,\ |\beta| \ll 1 \}$:
\end{center}
We start by proving that $\forall \ \varepsilon>0, \ \exists \ \beta_{\varepsilon} >0$ such that any $\varphi_{1,1,\beta}$ even ground state solution to \eqref{Phimaineq2} with $\alpha =1$ satisfies
\begin{equation}\label{u1}
\| \varphi_{1,1,\beta} - Q_{KdV}\|_{H^1} <\frac{\varepsilon}{2} \quad \quad \forall \ 0 <|\beta| < \beta_\varepsilon \ .
\end{equation}
This is clearly equivalent to proving that  for any sequence $(\beta_n)_{n \geq 1} \subset ]0,1]$ that tends to zero as $n \rightarrow + \infty$, and any sequence of associated even ground state solutions $(\varphi_{\beta_n})_{n \geq 1}$ to \eqref{Phimaineq2} with $\alpha=1$, it holds 
$$
\varphi_{\beta_n} \to Q_{KdV} \quad \text{in} \quad H^1(\mathbb{R}) \; .
$$
So, let be given a sequence $(\beta_n)_{n \geq 1} \subset \mathbb{R}^{\ast}_{+}$ that tends to zero as $n \rightarrow + \infty$, and a sequence of even ground state solutions $(\varphi_{\beta_n})_{n \geq 1}$ to \eqref{Phimaineq2} with $(\alpha,\beta)=(1,\beta_n)$.
Defining $\Phi_{\beta_n}$ in the following way 
\begin{equation} \label{uni1}
\Phi_{\beta_n} = \bigg ( \frac{K(Q_{KdV})}{K(\varphi_{\beta_n})} \bigg)^\frac{1}{3} \varphi_{\beta_n} \ ,
\end{equation}
  we get that $K(\Phi_{\beta_n}) = K(Q_{KdV}) >0$ and thus $\Phi_{\beta_n}$ is a solution to $M^{K(Q_{KdV})}_{1,1,\beta_n}$. According to Proposition \ref{Conv1}, it follows that $\Phi_{\beta_n} \rightarrow Q_{KdV}$ in $H^1(\mathbb{R})$ as $\beta_n$ tends to zero. 
Now, as \eqref{uni1} leads to
$$ I_{1,1,\beta_n}(\varphi_{\beta_n}) = \bigg ( \frac{K(\varphi_{\beta_n})}{K(Q_{KdV})} \bigg )^{\frac{2}{3}} I_{1,1,\beta_n}(\Phi_{\beta_n}) ,$$
\eqref{gssol} on the other hand shows that
$$ I_{1,1,\beta_n} (\varphi_{\beta_n}) = \frac{3}{2} \ K(\varphi_{\beta_n}) .$$
Combining the last two identities, we get the following
\begin{eqnarray}\nonumber
\displaystyle \frac{3}{2}  \frac{K({\varphi_{\beta_n}})}{K(Q_{KdV})}  =  \bigg ( \frac{K(\varphi_{\beta_n})}{K(Q_{KdV})} \bigg )^{\frac{2}{3}} \ \frac{I_{1,1,\beta_n} (\Phi_{\beta_n})} {K(Q_{KdV})} ,
\end{eqnarray}
or equivalently,
$$ \bigg ( \frac{K({\varphi_{\beta_n}})}{K(Q_{KdV})} \bigg )^{\frac{1}{3}} = \frac{2}{3} \ \frac{I_{1,1,\beta_n} (\Phi_{\beta_n})}{K(Q_{KdV})}  \ .
$$
Substituting the above identity in what follows, we obtain
\begin{equation}\label{uniq2}
\displaystyle \varphi_{\beta_n} = \bigg ( \frac{K(\varphi_{\beta_n})}{K(Q_{KdV})} \bigg)^\frac{1}{3} \Phi_{\beta_n} = \frac{2}{3} \ \frac{I_{1,1, \beta_n}(\Phi_{\beta_n})}{K(Q_{KdV})} \ \Phi_{\beta_n} \ . 
\end{equation} 
Yet, $\Phi_{\beta_n}$ converges to $Q_{KdV}$ in $H^1(\mathbb{R})$ and so $I_{1,1,\beta_n}(\Phi_{\beta_n}) \longrightarrow I_{1,1,0}(Q_{KdV})$, and $Q_{KdV}$ satisfies \eqref{gssol}, thus we reach
$$\varphi_{\beta_n} \longrightarrow Q_{KdV} \quad \text{in} \ H^1 $$
which proves \eqref{u1}. Now, we follow the argument of [\cite{Kenig Martel Robbiano 2011}, Proposition 3] to prove the uniqueness for $ |\beta|$ small enough. Let $\varphi^1_{\beta}$ and $\varphi^2_{\beta}$ be two ground state solutions to \eqref{Phimaineq2} with $\alpha=1$ and set $\tilde{w}= \varphi^1_{\beta} - \varphi^2_{\beta}$. We claim that $\tilde{w}=0$ as soon as $|\beta| $ is small enough. Thanks to \eqref{u1}, we can see that
$$ \|\tilde{w}\|_{H^1} \leq \|\varphi^1_{\beta} - Q_{KdV}\|_{H^1} + \|Q_{KdV} - \varphi^2_{\beta}\|_{H^1} \leq \varepsilon (\beta) , 
$$
with $ \varepsilon(y) \to 0 $ as $ y\to 0$. 
Moreover, $\tilde{w}$ satisfies
\begin{equation}\label{weq1}
\tilde{w} - \beta \ \mathcal{H} \partial_x \tilde{w} - \ \partial^2_x \tilde{w} = \frac{1}{2} (\varphi^1_{\beta})^2 - \frac{1}{2} (\varphi^2_{\beta})^2.
\end{equation}
We also define the operator $\Upsilon_{1,1,\beta}$ as
$$ \Upsilon_{1,1, \beta} = 1 - \beta \ \mathcal{H} \partial_x - \ \partial_x^2 - \varphi^1_{\beta} .$$
It was noticed in \cite{Kenig Martel Robbiano 2011} that the following crucial cancellation holds 
\begin{equation} \label{can}
\Upsilon_{1,1, \beta} \tilde{w}  = \frac{1}{2} (\varphi^1_{\beta})^2 - \frac{1}{2} (\varphi^2_{\beta})^2  - \varphi^1_{\beta} \tilde{w}= \frac{1}{2} \tilde{w}^2 \; .
\end{equation}
Let us now assume  that $\tilde{w} \neq 0$ and prove that it leads to a contradiction. Setting $w =\displaystyle \frac{\tilde{w}}{\|\tilde{w}\|_{H^1}}$ so that $\|w\|_{H^1}=1 $, it follows from \eqref{can} that 
\begin{equation} \label{cont1}
\|\Upsilon_{1,1, \beta} w\|_{L^2} \leq \frac{1}{\|\tilde{w}\|_{H^1}} \ \|\tilde{w}\|^2_{L^4} \leq \frac{1}{\|\tilde{w}\|_{H^1}} \ \|\tilde{w}\|^2_{H^1} = \|w\|_{H^1} \ \|\tilde{w}\|_{H^1} \leq \varepsilon(\beta). 
\end{equation} 
Hence by \eqref{cont1} and \eqref{u1}, we get
\begin{eqnarray}\nonumber
\displaystyle | \langle \Upsilon_{1,1, 0} w, w \rangle_{L^2} | &= & \bigg | \langle w - \partial_x^2 w - Q_{KdV} \ w , w \rangle_{L^2} \bigg | \nonumber \\
\displaystyle & \leq & \bigg | \langle \Upsilon_{1, 1, \beta} w, w \rangle_{L^2} + \langle \varphi^1_{\beta}\ w - Q_{KdV} \ w, w \rangle_{L^2} + \beta \langle \mathcal{H} \partial_x w, w \rangle_{L^2} \bigg | \nonumber \\
\displaystyle & \leq & \bigg | \langle \Upsilon_{1,1, \beta} w, w \rangle_{L^2} + \langle \varphi^1_{\beta} - Q_{KdV}, w^2 \rangle_{L^2} \bigg| + |\beta| \ \|D_x^{\frac{1}{2}} w\|^2_{L^2} \nonumber \\
\displaystyle & \le  & 2 \; \varepsilon(\beta) + |\beta|. \label{opcontr1}
\end{eqnarray}
Now, according to \cite{Bona et al 1987}, for any $u \in H^1(\mathbb{R})$ satisfying the orthogonality conditions $\langle u, Q_{KdV} \rangle_{L^2} = \langle u, Q'_{KdV} \rangle_{L^2} =0$, there exists $\gamma_0 >0$ such that 
\begin{equation} \label{orth}
\langle \Upsilon_{1,1, 0} u, u \rangle \geq \gamma_0 \|u\|^2_{H^1} .
\end{equation} 
Moreover, the coercivity of the operator $\Upsilon$ persists even with almost orthogonality conditions (see for instance \cite{Cote et al 2016}, Lemma 2.7). In other words, \eqref{orth} still holds with $ \gamma_0 $ replaced by $\gamma_0/2 $ for any $u \in H^1(\mathbb{R})$ satisfying
$$ |\langle u, Q_{KdV} \rangle|_{L^2} \leq \vartheta \ \|u\|_{H^1} \quad \text{and} \quad | \langle u, Q'_{KdV} \rangle|_{L^2} \leq \vartheta \ \|u\|_{H^1}  \quad \text{for \ some} \ \vartheta >0.$$
Before all else, the first orthogonality condition 
\begin{equation}\label{orth1}
\langle w, Q'_{KdV} \rangle =0
\end{equation} 
is straightforward as $w$ is even and $Q'_{KdV}$ is odd. The first almost orthogonality condition is thus a direct result.
For the second condition, we recall that differentiating \eqref{Phimaineq2c} with respect to $c$ and  taking $c=1$ leads to 
 $$\Upsilon_{1,1, 0} \ \partial_{c|{_{c = 1}}} Q_{c,KdV} = Q_{1,KdV}=Q_{KdV} \; . $$
 According to \eqref{KdV soliton2}, we have
$$
\partial_{c|{_{c = 1}}} Q_{c,KdV} =-\frac{3}{2} \tanh(x/2) \, {\rm sech}^2(x/2)=: {\mathcal V} \in H^\infty(\R) \; . 
$$
 Therefore, by \eqref{cont1} and Proposition \ref{Conv1}, we obtain
\begin{eqnarray}\nonumber
\displaystyle | \langle w, Q_{KdV|_{c=1}} \rangle_{L^2} | & = & \big | \langle w, \Upsilon_{1,1, 0} \,  {\mathcal V} \rangle_{L^2} \big | \nonumber \\
\displaystyle &=& \big | \langle \Upsilon_{1,1, 0} w,  {\mathcal V} \rangle_{L^2} \big | \nonumber \\
\displaystyle & \leq & \big | \langle \Upsilon_{1,1, \beta} w,  {\mathcal V} \rangle_{L^2} \big | + \big | \big\langle w  {\mathcal V} \ , \varphi^1_{\beta} - Q_{KdV} \big\rangle_{L^2} \big |+ \beta \ \big | \langle \mathcal{H} \partial_x w, {\mathcal V} \rangle_{L^2} \big| \nonumber \\
\displaystyle &\lesssim & \big ( \varepsilon (\beta) + \beta \big ) \|w\|_{H^1} \ . \nonumber
\end{eqnarray}
Thus, for $|\beta|$ small enough, $w$ is almost orthogonal to $Q_{KdV}$, and so we deduce that
$$ | \langle \Upsilon_{1,1, 0} w, w \rangle_{L^2} | \geq \frac{\gamma_0}{2} \|w\|^2_{H^1} = \frac{\gamma_0}{2} $$
which contradicts \eqref{opcontr1}. This  asserts that $\tilde{w}=0$ and proves the uniqueness of the ground state to \eqref{Phimaineq2}, with $ (\alpha,\beta)\in \{1\} \times ]-\tilde{\beta}_0,  \tilde{\beta}_0 [  $ for some $ \tilde{\beta}_0>0$. In the sequel, we denote this unique ground state by $ Q_{1,1,\beta} $.
\\
Now that uniqueness is established, we can repeat the argument in the beginning of the proof of this proposition to prove the continuity of the map $ \beta \mapsto Q_{1,1,\beta}$. Indeed, let $ \tilde{\beta}\in ]-\tilde{\beta}_0,\tilde{\beta}_0[ $,  let $ (\beta_n) \in ]-\tilde{\beta}_0,\tilde{\beta}_0[$ with $ \beta_n \to \tilde{\beta} $,  and let $Q_{1,1,\beta_n} $ be the unique associated even ground state to  \eqref{Phimaineq2} with 
$ (\alpha,\beta)=(1,\beta_n) $.\\
Similarly as in \eqref{uni1}, we define $\Phi_{\beta_n}$ as
\begin{equation} \label{uni2}
\Phi_{\beta_n} = \bigg ( \frac{K(Q_{1,1,\tilde{\beta}})}{K(Q_{1,1,\beta_n})} \bigg)^\frac{1}{3} Q_{1,1,\beta_n} \ ,
\end{equation}
  so that $K(\Phi_{\beta_n}) = K(Q_{1,1,\tilde{\beta}}) >0$ and  $\Phi_{\beta_n}$ is a solution to $M^{K(Q_{1,1,\tilde{\beta}})}_{1,1,\beta_n}$. As in the proof of  Proposition \ref{Conv1}, one can easily check that $(\Phi_{\beta_n})$ is a minimizing sequence of 
  $M^{K(Q_{1,1,\tilde{\beta}})}_{1,1,\tilde{\beta}}$ and  Proposition \ref{minimizer} ensures that it converges up to a subsequence  $(\Phi_{\beta_{n_k}}) $ towards a solution $ \varphi $  to $M^{K(Q_{1,1,\tilde{\beta}})}_{1,1,\tilde{\beta}}$ in $ H^1(\mathbb{R}) $.
  Now, as in \eqref{uniq2}, it holds 
\begin{equation}
\displaystyle Q_{1,1,\beta_n} = \bigg ( \frac{K(Q_{1,1,\beta_n})}{K(Q_{1,1,\tilde{\beta}})} \bigg)^\frac{1}{3} \Phi_{\beta_n} = \frac{2}{3} \ \frac{I_{1,1, \beta_n}(\Phi_{\beta_n})}{K(Q_{1,1,\tilde{\beta}})} \ \Phi_{\beta_n} \  \nonumber
\end{equation} 
and since $I_{1,1, \beta_n}(\Phi_{\beta_n})\to  M^{K(Q_{1,1,\tilde{\beta}})}_{1,1,\tilde{\beta}}=I_{1,1,\tilde{\beta}}(Q_{1,1,\tilde{\beta}})$ and 
$ Q_{1,1,\tilde{\beta}}$ satisfies \eqref{gssol}, we deduce that $ Q_{1,1,\beta_{n_k}} \to \varphi $ in $ H^1(\mathbb{R}) $. This ensures that $ \varphi $ satisfies \eqref{Phimaineq2} with 
$ (\alpha,\beta)=(1,\tilde{\beta}) $ and thus is an even ground state to this equation. Finally, by the proved uniqueness result, it follows that 
 $\varphi=Q_{1,1,\tilde{\beta}} $, and by the uniqueness of the possible limits in $H^1(\mathbb{R})$ of its subsequences, we deduce that $  Q_{1,1,\beta_n}  \rightarrow Q_{1,1,\tilde{\beta}}$  in $ H^1(\mathbb{R}) $. 
\begin{center}
\underline{\bf{Case 2}:} $\{0\le \alpha  \ll 1, \ \beta= 1 \}$:
\end{center}
This case is similar to the preceding one only up to small differences. First, we claim that $\forall \ \varepsilon>0, \ \exists \ \alpha_{\varepsilon} >0$ such that any even ground state solution $\psi_{1,\alpha,1}$ to \eqref{Phimaineq2} with $\beta =1$ satisfies
\begin{equation}\label{u2}
\| \psi_{1,\alpha,1} - Q_{BO}\|_{H^1} <\frac{\varepsilon}{2} \quad \quad \forall \ 0 <\alpha < \alpha_\varepsilon .
\end{equation}
The proof of \eqref{u2} follows exactly the same lines as the proof of \eqref{u1}, but this time making use of Proposition \ref{pconv2} instead of Proposition \ref{Conv1}. It will thus be omitted. 

 Let us now tackle the proof of  the uniqueness result. Let $\varphi^1_{\alpha}$ and $\varphi^2_{\alpha}$ be two ground state solutions to \eqref{Phimaineq2} with $\beta=1$. Set $\tilde{w}= \varphi^1_{\alpha} - \varphi^2_{\alpha}$; we will seek a contradiction argument to prove that $\tilde{w}=0$ for $\alpha$ small enough. Thanks to \eqref{u2}, we can see that
$$ \|\tilde{w}\|_{H^1} \leq \|\varphi^1_{\alpha} - Q_{BO}\|_{H^1} + \|Q_{BO} - \varphi^2_{\alpha}\|_{H^1} \leq \varepsilon (\alpha) \; ,$$
with $ \varepsilon(y)\to 0 $ as $ y\to 0 $.
Moreover, $\tilde{w}$ satisfies
\begin{equation}\label{weq}
\tilde{w} - \mathcal{H} \partial_x \tilde{w} -\alpha \ \partial^2_x \tilde{w} = \frac{1}{2} \ (\varphi^1_{\alpha})^2 - \frac{1}{2} \ (\varphi^2_{\alpha})^2 .
\end{equation}
We also define the operator $\Gamma$ to be
$$ \Gamma_{1,\alpha,1} = 1 - \mathcal{H} \partial_x - \alpha \ \partial_x^2 - \varphi^1_{\alpha} $$
and check that 
\begin{equation}\label{weq2}
 \Gamma_{1,\alpha,1} \tilde{w} =  \frac{1}{2} \ (\varphi^1_{\alpha})^2 - \frac{1}{2} \ (\varphi^2_{\alpha})^2 - (\varphi^1_{\alpha}) \tilde{w} = \frac{1}{2} \tilde{w}^2 .
 \end{equation}
Let us now assume  that $\tilde{w} \neq 0$ and set $w =\displaystyle \frac{\tilde{w}}{\|\tilde{w}\|_{H^\frac{1}{2}}}$ so that 
$\|w\|_{H^\frac{1}{2}}=1$. Thanks to \eqref{weq2} together with the embedding $ H^\frac{1}{2}(\mathbb{R})  \hookrightarrow L^4(\mathbb{R})$, we get
\begin{equation} \label{cont2}
\|\Gamma_{1,\alpha,1}w\|_{L^2}  \leq \frac{1}{\|\tilde{w}\|_{H^\frac{1}{2}}} \ \|\tilde{w}\|^2_{L^4} \leq \|w\|_{H^\frac{1}{2}} \ \|\tilde{w}\|_{H^\frac{1}{2}} \leq \varepsilon(\alpha).
\end{equation}
Hence by \eqref{cont2} and \eqref{u2}, we obtain
\begin{eqnarray}\nonumber
\displaystyle  \langle \Gamma_{1,0,1} w, w \rangle_{L^2}  &= & \langle w - \mathcal{H} \partial_x w - (Q_{BO}) w , w \rangle_{L^2}  \nonumber \\
\displaystyle & = &  \langle \Gamma_{1, \alpha, 1} w, w \rangle_{L^2} + \langle (\varphi^1_{\alpha}) w - (Q_{BO}) w, w \rangle_{L^2} + \alpha \langle \partial_x^2 w, w \rangle_{L^2}  \nonumber \\
\displaystyle & = &   \langle \Gamma_{1, \alpha, 1} w, w \rangle_{L^2} + \langle \varphi^1_{\alpha} - Q_{BO}, w^2 \rangle_{L^2} - \alpha \ \|\partial_x w\|_{L^2} \nonumber \\
\displaystyle & \le & \varepsilon(\alpha)+ \|\varphi^1_{\alpha} - Q_{BO}\|_{H^1} \|w\|_{L^2}^2 \nonumber \\
&\le  & 2 \; \varepsilon(\alpha) . \label{opcontr2} 
\end{eqnarray}
Note that we have used in a crucial way the fact that $ \alpha\ge 0 $ since otherwise we do not have any control on $ \|\partial_x w\|_{L^2}$. Now according to \cite{Bennett et al. 1983} or \cite{Weinstein 1987}, there exists $\eta_0 >0$ such that for any $u \in H^1(\mathbb{R})$ satisfying the orthogonality conditions $\langle u, Q_{BO} \rangle_{L^2} = \langle u, Q'_{BO} \rangle_{L^2} =0$, we have
\begin{equation} \label{ortho}
 \langle \Gamma_{1,0, 1} u, u \rangle \geq \eta_0 \|u\|^2_{H^\frac{1}{2}} 
\end{equation}
and again thanks to \cite{Cote et al 2016}, \eqref{ortho} still holds for $u \in H^\frac{1}{2}$ satisfying almost orthogonality conditions. To start with, the first orthogonality condition 
\begin{equation}\label{orth2}
\langle w, Q^{'}_{BO} \rangle =0
\end{equation} 
is direct (and thus the first almost orthogonality condition too) since $w$ is even and $Q^{'}_{BO}$ is odd. For the second condition, we proceed as above by noticing that $\Gamma_{1,0,1} {\mathcal W}=Q_{BO} $ where according to \eqref{Phimaineq2c} and \eqref{BO soliton2}, 
$$
 {\mathcal W}=\partial_{c|{_{c = 1}}} Q_{c,BO}=-\frac{8x}{(1+x^2)^2} \; .
$$
Therefore, \eqref{cont2} and Proposition \ref{pconv2} lead to 
\begin{eqnarray}\nonumber
\displaystyle | \langle w, Q_{BO} \rangle_{L^2} | & = & \big | \langle w, \Gamma_{1,0, 1} {\mathcal W}\rangle_{L^2} \big | \nonumber \\
\displaystyle &=& \big | \langle \Gamma_{1,0,1} w, {\mathcal W} \rangle_{L^2} \big | \nonumber \\
\displaystyle & \leq & \big | \langle \Gamma_{1,\alpha,1} w,{\mathcal W} \rangle_{L^2} \big | + \big | \langle w \, {\mathcal W} \ , \varphi^1_{\alpha} - Q_{BO}  \rangle_{L^2} \big | + \alpha \ \big | \langle w, \partial_x^2{\mathcal W} \rangle_{L^2} \big| \nonumber \\
\displaystyle &\lesssim &   \Bigl( \varepsilon (\alpha) + \alpha\Bigr) \|w\|_{H^\frac{1}{2}} \; . \nonumber
\end{eqnarray}
Thus, for $\alpha$ small enough, $w$ is almost orthogonal to $Q_{BO}$, and so we conclude that 
$$  \langle \Gamma_{1,0,1} w, w \rangle_{L^2} \geq \frac{\eta_0}{2} \|w\|^2_{H^\frac{1}{2}} = \frac{\eta_0}{2}  ,$$
which contradicts \eqref{opcontr2}. This hence asserts that $\tilde{w}=0$ and uniqueness is proved. In the sequel, we let $ Q_{1,\alpha,1} $ denote this unique even ground state.
\\
Finally, the continuity of $ \alpha\mapsto Q_{1,\alpha,1} $ follows exactly the same lines as in {\bf{Case 1}} and will thus be disregarded. 
\end{proof}
\subsubsection{Orbital stability}
To prove the orbital stability of an even ground state solution in $H^1(\mathbb{R})$, we recall the following proposition; check \cite{Bouard 2005}. 
\begin{proposition}\label{orbital}
Set the operator $\Upsilon_{c,\alpha, \beta}$ associated to the second derivative of the action functional $E+cV$ and linearized at $\varphi_{c,\alpha, \beta} \in H^1(\mathbb{R})$ solution to \eqref{Phimaineq2}; where $\Upsilon_{c,\alpha, \beta}$ is defined by
$$\Upsilon_{c,\alpha, \beta} := c - \alpha \partial_x^2 - \beta \mathcal{H} \partial_x - \varphi_{c,\alpha,\beta} .$$ 
Assume that there exists $\gamma >0$ such that
$$\langle \Upsilon_{c,\alpha,\beta} u, u \rangle _{L^2} \geq \gamma \|u\|^2_{H^1}$$
for any $u \in H^1(\mathbb{R})$ satisfying the orthogonality conditions
$$ \langle u, \varphi_{c,\alpha,\beta} \rangle_{L^2} = \langle u, \varphi'_{c,\alpha,\beta} \rangle = 0$$
then $\varphi_{c,\alpha,\beta}$ is orbitally stable in $H^1(\mathbb{R})$. More precisely, there exists $ K>0  $ and $ \varepsilon_0>0 $ such that if $ u_0\in H^1(\mathbb{R})$ satisfies
$$
\|u_0- \varphi_{c,\alpha,\beta}\|_{H^1} <\varepsilon 
$$
for some $ 0<\varepsilon<\varepsilon_0 $ then 
$$
\sup_{t\in\R} \inf_{y\in\R} \| u(t)-\varphi_{c,\alpha,\beta}(\cdot-y)\|_{H^1} < K \varepsilon \; .
$$

\end{proposition}

We are now in position to prove the orbital stability of the even ground state solutions $\varphi_{1,1,\beta}$ of \eqref{Phimaineq2} with $\alpha=1$, then that of $\varphi_{1,\alpha, 1}$ for $\beta=1$ in \eqref{Phimaineq2}.
\begin{center}
\underline{\bf{Case 1}:} Orbital stability of $\varphi_{1,1,\beta}$
\end{center}
To begin with, we note that thanks to Proposition \ref{Conv1}, any $u\in H^1$ satisfying $\langle u, Q_{1,1,\beta} \rangle_{L^2}= \langle u, Q'_{1,1,\beta} \rangle_{L^2} =0$ is almost orthogonal in $L^2$ to $Q_{KdV}$ and $Q'_{KdV}$ for $|\beta|$ small enough and thus \eqref{orth} holds with $ \gamma_0 $ replaced by $ \gamma_0/2$. Hence,
\begin{eqnarray}
\displaystyle \langle \Upsilon_{1,1, \beta} u, u \rangle_{L^2} &=& \langle \Upsilon_{1,1,0} u, u \rangle_{L^2} - \langle Q_{1,1,\beta} - Q_{KdV} , u^2 \rangle_{L^2} - \beta \langle \mathcal{H} \partial_x u, u \rangle_{L^2} \nonumber \\
\displaystyle & \geq & \frac{\gamma_0}{2} \|u\|^2_{H^1} - \varepsilon(\beta) \|u\|^2_{H^1} - |\beta| \ \|D_x^\frac{1}{2} u\|^2_{L^2} \nonumber\\
\displaystyle &\geq & \bigg( \frac{\gamma_0}{2} - \varepsilon(\beta) - \varepsilon (|\beta|) \bigg) \|u\|^2_{H^1} \label{coercivity}
\end{eqnarray}
where $\varepsilon(\beta), \  \varepsilon (|\beta|) \longrightarrow 0$ as $\beta \rightarrow 0$. Eventually, thanks to Proposition \ref{orbital}, we deduce that there exists $r_0>0$ such that $\varphi_{1,1,\beta}$ is orbitally stable in $H^1(\mathbb{R})$ for any $|\beta|< r_0$.
\\
Lastly, since $u$ is a solution of \eqref{maineq2}  with $ \alpha=1 $ if and only if $v(t,x) = \lambda^{2} \ u(\lambda^{3}\ t, \lambda \ x)$ is a solution to \eqref{maineq} with $ \gamma=|\lambda| \beta $, there is  an equivalence between the uniqueness and orbital stability of $Q_{1,1,\beta}$ and the ones of $Q_{\lambda^2,1,|\lambda|\beta}$. This leads to the  uniqueness and orbital stability of $ Q_{c,1,\gamma} $ with $ \gamma \in \mathbb{R}^* $ as soon as 
$ c>r_0^{-2} \gamma^2 $.
\begin{center}
\underline{\bf{Case 2}:} Orbital stability of $\varphi_{1,\alpha,1}$
\end{center}
Again, we first note that thanks to Proposition \ref{pconv2}, any $u\in H^\frac{1}{2}(\mathbb{R})$ satisfying $\langle u, Q_{1,\alpha,1} \rangle= \langle u, Q'_{1,\alpha,1} \rangle =0$ is almost orthogonal in $L^2$ to $Q_{BO}$ and $Q'_{BO}$ for $\alpha>0$ small enough and thus \eqref{ortho} holds with $ \eta_0 $ replaced by $ \eta_0/2$.
Hence
\begin{eqnarray}\nonumber
\displaystyle \langle \Upsilon_{1,\alpha,1} u, u \rangle_{L^2} &=& \langle \Upsilon_{1,0,1} u, u \rangle_{L^2} - \langle Q_{1,\alpha,1} - Q_{BO} , u^2 \rangle_{L^2} + \alpha \|\partial_x u\|^2_{L^2} \nonumber \\
\displaystyle & \geq & \frac{\eta_0}{2} \|u\|^2_{H^\frac{1}{2}} - \varepsilon(\alpha) \|u\|^2_{L^2} + \alpha \|\partial_x u\|^2_{L^2} \nonumber\\
\displaystyle &\geq & \min \bigg \{ \frac{\eta_0}{2} - \varepsilon(\alpha), \alpha \bigg\} \|u\|^2_{H^1} \nonumber
\end{eqnarray}
where $\varepsilon(\alpha) \longrightarrow 0$ as $\alpha \rightarrow 0$. Eventually, thanks to Proposition \ref{orbital}, we deduce that there exists $\tilde{r}_0>0$ such that $Q_{1,\alpha,1}$ is orbitally stable in $H^1 (\mathbb{R})$ for any $\alpha \in ]0, \tilde{r}_0[$.
\\
Now, we return to the uniqueness and orbital stability of solution to \eqref{maineq}. Let us fix $ \gamma>0 $. Since $u$ is a solution to \eqref{maineq2} with $\beta=1 $ if and only if $v(t,x)= \gamma^2 \alpha u (\gamma^3 \alpha^2  t, \gamma \alpha  x)$ satisfies \eqref{maineq}, then the uniqueness and orbital stability of $Q_{1,\alpha,1}$ is equivalent to the uniqueness and orbital stability of $Q_{\gamma^2 \alpha, 1,\gamma}$. This leads to the uniqueness and orbital stability of $ Q_{c,1,\gamma}$ with $ \gamma>0 $ as soon as
 $ 0<c<\tilde{r}_0 \gamma^2 $, and eventually completes the proof of Theorem \ref{Thm1}.
 \section{Decay of the ground state solutions} \label{section 4}
In this last section, we prove some decay estimates for any even $H^1$-solution to \eqref{Phimaineq} as well as its two first derivatives. We follow the arguments in \cite{bonali}. Note that such decay  estimates were already proved in \cite{bonachen}, except for the derivatives.  It is also worth noting that these  decay estimates are an important requirement to derive sharp properties of the flow of the equation in the neighborhood of the solitary waves.
\begin{proposition}\label{prodecay}
Let $\gamma\in\R $. Given that $ c>0 $ in the case $ \gamma>0  $ and $ c>\gamma^2/4 $ in the case $ \gamma<0 $, the even  solutions to \eqref{Phimaineq} satisfy 
$$
 |x^2\,  \varphi_{c,1,\gamma}|+  |x^3 \, \varphi_{c,1,\gamma}' (x)|+ |x^3 \, \varphi_{c,1,\gamma}''|\in L^\infty(\mathbb{R})\;.
 $$
\end{proposition}
\begin{proof}
We consider only the case $ c=1 $ and $ \gamma >-2 $. Indeed, as already mentioned in the beginning of Section \ref{section 2},  $ \varphi $ is a solution to \eqref{Phimaineq} with $ c=1 $ and $ \gamma=\beta$ if and only if $ \phi(\cdot)= \lambda^2 \varphi (\lambda\cdot) $ is a solution to \eqref{Phimaineq} with $ c=\lambda^2 $ and $ \gamma = |\lambda| \beta $. Therefore, the decay estimate  for $ c=1 $ and $\gamma >-2 $ ensures the same decay estimate for any speed $ c>0 $ in the case $ \gamma>0 $ and for any speed $c>\gamma^2/4 $ in the case $ \gamma<0 $.
Now, let us denote by $ \varphi_{\gamma} $ an even $H^1$-solution to   \eqref{Phimaineq} with $ c=1 $ and $ \gamma >-2 $. We know that $ \varphi_\gamma\in H^\infty(\mathbb{R}) $ and according to \eqref{Phimaineq} it holds 
\begin{equation} \label{four2}
\widehat{\varphi_{\gamma}} (\xi) =  \frac{1}{2( 1+  \xi^2 + \gamma |\xi|)} \ \widehat{\varphi_{\gamma} ^2} (\xi) .
\end{equation}
Therefore, for $ k=0,1$, $\varphi_{\gamma}^{(k)} (x) = G^{(k)} \ast \varphi_{\gamma}^2 (x) $ with $G^{(k)}$ given by
$$ G^{(k)}(x) = \frac{1}{2\sqrt{2 \pi}} \int_{\mathbb{R}} \frac{ (i \xi)^{k}  \ e ^{i x \xi}}{1+ \alpha \xi^2 + \gamma |\xi|} d\xi \ .$$
Since $\displaystyle \frac{1}{2( 1+  \xi^2 + \gamma |\xi|)}$ belongs to $ H^1(\R)$, the hypotheses of (cf. \cite{bonali}, Theorem 3.1.2, Page 384) are clearly fulfilled. Thus according to this theorem, there exists $ 0< l \leq 1 $ such that $ |x|^l \varphi_\gamma \in L^\infty(\mathbb{R})$. Since  $ \varphi_\gamma $ belongs to $L^2(\mathbb{R}) $, it follows that $ \varphi_\gamma^2 $ belongs to $L^1(\mathbb{R}) $ and decays at least as $ |x|^{-2l}$. We are going to prove that $ G $ decays at least as $ |x|^{-2} $ and thus also belongs to $L^1(\mathbb{R}) $.  Assuming this decay for a while, the identity $\varphi_{\gamma}  = G \ast \varphi_{\gamma}^2  $  ensures that $ \varphi_\gamma $ decays at least as $ |x|^{-2l}$; reiterating this argument for a finite number of times, we obtain that  $ \varphi_\gamma $ decays at least as $|x|^{-2} $. 
Therefore, it remains to prove the decay of $ G$. We start by noticing  that since $\displaystyle \frac{1}{1+ \alpha \xi^2 + \beta |\xi|}$ is a real and even function, then so is its inverse Fourier. Hence, 
$$
\displaystyle G(x) = \frac{1}{2\sqrt{2 \pi}} \int_{\mathbb{R}} \frac{\cos(x \xi)}{1+ \alpha \xi^2 + \gamma |\xi|} \ d\xi = \frac{1}{\sqrt{2 \pi}} \int_{\mathbb{R}^+} \frac{\cos (x \xi)}{1+ \alpha \xi^2 + \gamma \xi} \ d\xi .\nonumber
$$
Fix $x$, and set $g(z) = g(x,z) =\displaystyle \frac{e^{i |x| z}}{1+ z^2 + \gamma z} $. We notice that the roots of $p(z)=1+ z^2 + \gamma z $ are given when $\gamma \ge 2 $ by $ \frac{-\gamma\mp \sqrt{\gamma^2-4}}{2} $ and when $ |\gamma|< 2 $ by $ \frac{-\gamma\mp i \sqrt{4-\gamma^2}}{2} $. Therefore, $ g $ is analytic on the open convex set 
$$
\Lambda_{\theta}:=\{ z\in {\mathbb C} \, /\, \Re(z)>\theta_1 \text{ and } |\Im z |<\theta_2 \} 
$$
with $ (\theta_1,\theta_2)=(\frac{-\gamma+\sqrt{\gamma^2-4}}{2},1) $ when $ \gamma\ge 2 $,   $ (\theta_1,\theta_2)=(-\gamma/2, 1) $ when $ 0<\gamma< 2 $ and $ (\theta_1,\theta_2)= (-1,\frac{\sqrt{4-\gamma^2}}{2}) $ when $ -2<\gamma\le 0 $.
We can thus use Cauchy's Integral Theorem along the closed path $\delta_R := \delta_{1,R} + \delta_{2,R} + \delta_{3,R} + \delta_{4,R}$ , where $\delta_{1,R}= [0,R]$, $\delta_{2,R} = \{R+i \ t; \ t \in [0,\frac{\theta}{2}]$ \}, $\delta_{3,R} = \{ t+i \frac{\theta}{2} ; \ t \in [R, 0]\}$, and $\delta_{4,R}= \{ i \ t; \ t \in [\frac{\theta}{2}, 0] \}$, to get
\begin{eqnarray}\nonumber
\displaystyle \oint_{\delta_R} g(z) \ dz &=& \int_{\delta_{1,R}} g(z) \ dz + \int_{\delta_{2,R}} g(z) \ dz + \int_{\delta_{3,R}} g(z) \ dz + \int_{\delta_{4,R}} g(z) \ dz \nonumber \\
\displaystyle &:=& I_{1,R} + I_{2,R} + I_{3,R} + I_{4,R} \nonumber \\
\displaystyle &=& 0. \nonumber
\end{eqnarray}
Since $\displaystyle \frac{\cos (x \xi)}{1+ \alpha \xi^2 + \gamma \xi}$ is integrable on $ \mathbb{R}^*_+$, we have  
$$\sqrt{2\pi} \,  G(x) = \lim_{R \rightarrow +\infty} \Re(I_{1,R})= \lim_{R \rightarrow +\infty} 
\Re\Bigl( - I_{2,R} - I_{3,R} - I_{4,R}\Bigr) .$$
We recall that for a piecewise $C^1$-path 
$$\gamma : [t_1,t_2] \ni t \mapsto \gamma(t) \in \mathbb{C} , $$ 
we have
 $$ \int _\gamma g(z) \ dz = \int_{t_1}^{t_2} g(\gamma(t)) \ \gamma'(t) \ dt .$$
Starting with $I_2$,
\begin{eqnarray}\nonumber
\displaystyle I_{2,R}  &=& i \int_0^{\frac{\theta}{2}} \frac{e^{i |x| (R + it)}}{1 + (R +it)^2 +\gamma (R+it)} \ dt = i e^{-|x|t} \int_0^{\frac{\theta}{2}} \frac{e^{i |x| R }}{1 + (R +it)^2 +\gamma (R+it)} \ dt \nonumber.
\end{eqnarray}
Using the fact that $|e^{i |x| R}| \leq 1$, integrating over $[0,\frac{\theta}{2}]$, then taking the limit $R \rightarrow +\infty$, we end up with 
\begin{equation}\label{I2}
\lim_{R \rightarrow + \infty} |I_{2,R} | = 0 .
\end{equation}
As for $I_3$, 
$$ I_{3,R} = \int_R^0 \frac{e^{i |x| (t + i\frac{\theta}{2})}}{1+(t+i\frac{\theta}{2})^2 + \gamma (t + i\frac{\theta}{2})} \ dt = -e^{-\theta|x|/2}  \int_0^R  \frac{e^{i |x| t }}{1+(t+i\frac{\theta}{2})^2 + \gamma (t + i\frac{\theta}{2})} \ dt ,$$
so there exists $C=C(\gamma)>0$ independent of $R$ such that
\begin{equation}\label{I3}
|I_{3,R}| \leq C \ e^{-\frac{\theta |x|}{2}} .
\end{equation}
Lastly for $I_{4,R} $, we get
\begin{eqnarray}\nonumber
\displaystyle I_{4,R} &=& i \int_{\frac{\theta}{2}}^0 \frac{e^{i |x| (i t)}}{1 + (it)^2 +\gamma i t} \ dt = -i \int_0^{\frac{\theta}{2}} \frac{e^{-|x| t}}{(1-t^2) + \gamma i t} \nonumber \\
\displaystyle &=& -i \int_0^{\frac{\theta}{2}} \frac{(1-t^2) - i \gamma t}{(1-t^2)^2 + \gamma^2 t^2} e^{-|x| t}\, dt \nonumber
\end{eqnarray}
and so
$$ \Re(I_{4,R} )= - \int_0^\frac{\theta}{2} \frac{\gamma t \ e^{-|x| t}}{(1-t^2)^2 + \gamma^2 t^2} \ dt .$$
Performing the change of variable $y=|x| t$, we get 
$$ \Re(I_{4,R} ) = -\frac{\gamma}{|x|^2} \int_0^{\frac{\theta|x|}{2}} \frac{y e^{-y}}{(1-\frac{y^2}{|x|^2})^2 + \gamma ^2 \frac{y^2}{|x|^2}} \ dy .$$
But, since $\displaystyle 0 \leq y \leq\frac{ \theta|x|}{2}$ then $\displaystyle \frac{1}{(1-\frac{y^2}{|x|^2})^2 + \gamma ^2 \frac{y^2}{|x|^2}} \leq \frac{1}{(1-\frac{y^2}{|x|^2})^2} \leq \frac{1}{1-\frac{\theta^2}{4}}\le \frac{4}{3}$. Besides, the fact that $\displaystyle \int y e^{-y} \ dy = e^{-y} (-y-1)$  ensures that 
\begin{equation}\label{I4}
|\Re(I_{4,R} )| \leq  C\,  \frac{\gamma}{|x|^2}\; .
\end{equation}
Finally, we get by \eqref{I2}, \eqref{I3}, and \eqref{I4} that 
\begin{equation} \label{decay}
|G(x)| \lesssim \frac{|\gamma|}{|x|^2} +C(\gamma) e^{-\frac{\theta |x|}{2}} 
\end{equation}
and so the decay of the even ground state $ \varphi_\gamma $ is given at least by $C_{\gamma} |x|^{-2}$, and that of $Q'_{\gamma}$ is given by $C_{\gamma} |x|^{-3}$.

In fact, to obtain the decay on $\varphi_\gamma' $, we are going to prove that $ G' $ decays at least as $ |x|^{-3}$. For this, we proceed exactly in the same way. We first notice that 
$$
\displaystyle G'(x) = \frac{-1}{\sqrt{2 \pi}} \int_{\mathbb{R}} \frac{\xi \sin(x \xi)}{1+ \alpha \xi^2 + \gamma |\xi|} \ d\xi  \nonumber
$$
 and for $x$ fixed, we set $\tilde{g}(z) = \tilde{g}(x,z) =\displaystyle \frac{-ze^{i |x| z}}{1+ z^2 + \gamma z} $ and we integrate the latter along $ \delta_R $.  Of course, $ \displaystyle \frac{\xi \sin(x \xi)}{1+ \alpha \xi^2 + \gamma |\xi|}$ does not belong to $L^1(\mathbb{R}^*_+) $. As the latter is no longer integrable on $\mathbb{R}$, we resort to oscillatory integrals where thanks to an integration by parts, it is direct to check that its generalized integral converges on $ \mathbb{R}_+ $. Therefore as above, it holds 
$$\sqrt{2\pi} \,  G'(x) = \lim_{R \rightarrow +\infty} \Im  \Bigl(\int_{\delta_{1,R}}  \tilde{g}(z) \ dz\Bigr)= \lim_{R \rightarrow +\infty} 
\Im\Bigl( - \int_{\delta_{2,R}}  \tilde{g}(z) \ dz- \int_{\delta_{3,R}}  \tilde{g}(z) \ dz - \int_{\delta_{4,R}}  \tilde{g}(z) \ dz\Bigr) \; .
$$ 
We set $ \displaystyle \tilde{I}_{k,R}=\int_{\delta_{k,R}}  \tilde{g}(z) \ dz $ for $ k=2,3,4$. We notice that $ \tilde{I}_{2,R} $ can be treated exactly as $ I_{2,R} $ to obtain $ \displaystyle \lim_{R\to +\infty} \tilde{I}_{2,R}=0$ and that 
$$ \tilde{I}_{3,R}  = -e^{-\theta|x|/2}  \int_0^R  \frac{(t+i\frac{\theta}{2})e^{i |x| t }}{1+(t+i\frac{\theta}{2})^2 + \gamma (t + i\frac{\theta}{2})} \ dt .$$
Integrating by parts, we get that there exists $C=C(\gamma)>0$ independent of $R$ such that for $ |x|\ge 1$, 
\begin{equation}\label{I3'}
|\tilde{I}_{3,R}| \leq C \,  e^{-\frac{\theta |x|}{2}} .
\end{equation}
As for $\tilde{I}_{4,R}$, we write 
$$
\displaystyle \tilde{I}_{4,R} = i \int_0^{\frac{\theta}{2}} \frac{(it) e^{-|x| t}}{(1-t^2) + \gamma i t} \nonumber \\
\displaystyle = - \int_0^{\frac{\theta}{2}} \frac{ t ( (1-t^2) - i \gamma t) }{(1-t^2)^2 + \gamma^2 t^2} e^{-|x| t}\, dt \nonumber
$$
so that 
$$ \Im(\tilde{I}_{4,R} )=  \int_0^\frac{\theta}{2} \frac{\gamma t^2 \ e^{-|x| t}}{(1-t^2)^2 + \gamma^2 t^2} \ dt .$$
Performing the change of variable $y=|x| t$, we get that 
$$|\Im(\tilde{I}_{4,R} )| \lesssim \gamma |x|^{-3} \ .$$
 Gathering the above estimates, we obtain the desired decay for $ G' $, and the identity  $\varphi_{\gamma}' = G' \ast \varphi_{\gamma}^2  $ together with  the decay of $\varphi_{\gamma}$ ensures that  $\varphi_{\gamma}' $ decays at least as $|x|^{-3} $. Consequently, the decay of $\varphi_{\gamma}''$ follows by noticing that 
 $\varphi_{\gamma}''= 2 G' \ast (\varphi_{\gamma}\varphi_{\gamma}') $.
\end{proof}

\end{document}